\documentclass[11pt,reqno,a4paper]{amsart}

\usepackage{amsmath,amsfonts,amssymb,amsthm,mathrsfs}
\usepackage{rotating,pgf,tikz,float}
\usepackage{graphicx,color}
\usetikzlibrary{automata,positioning,arrows.meta,shapes.misc}

\usepackage[hmargin=3.0cm,bmargin=2.9cm,tmargin=3.4cm]{geometry}

\newtheorem{theorem}{Theorem}[section]
\newtheorem{lemma}[theorem]{Lemma}
\newtheorem{proposition}[theorem]{Proposition}

\newtheorem*{conjecture*}{Conjecture}

\theoremstyle{definition}
\newtheorem{definition}[theorem]{Definition}

\theoremstyle{remark}
\newtheorem{remark}[theorem]{Remark}

\numberwithin{equation}{section}
\numberwithin{figure}{section}

\newcommand{\R}{\mathbb{R}}
\newcommand{\N}{\mathbb{N}}

\newcommand{\dx}{\; {\rm d}x}

\renewcommand{\d}{{\rm d}}

\DeclareMathOperator{\dist}{dist}
\DeclareMathOperator{\supp}{supp}

\DeclareMathOperator{\interior}{int}
\DeclareMathOperator{\EffDeg}{eff\,deg}

\newcommand{\partition}{\mathcal{P}}
\newcommand{\partitionk[1]}[]{\mathcal{P}_{#1}}

\newcommand{\energyc}{\Lambda^C}
\newcommand{\energyr[1]}[]{\Lambda^{#1}}
\newcommand{\cheeger[1]}[]{h_{#1}}
\newcommand{\robinlambda[1]}[]{\lambda_1^{#1}}
\newcommand{\roptenergy[2]}[]{\mathcal{L}^{#1}_{#2}}
\newcommand{\perim}{{\rm Per}_\Gamma}
\newcommand{\robinform[1]}[]{Q_{#1}}
\newcommand{\dirichletlambda}{\lambda_1^D}
\newcommand{\energyd}{\Lambda^D}
\newcommand{\doptenergy[1]}[]{\mathcal{L}^D_{#1}}

\newcommand{\VertexSet}{\mathcal{V}}
\newcommand{\EdgeSet}{\mathcal{E}}
\newcommand{\BV}{{\rm BV}}

\title{Cheeger cuts and Robin spectral minimal partitions of metric graphs}



\author{James B.~Kennedy}
\author{Jo\~ao P.~Ribeiro}

\address{Departamento de Matem\'atica, Faculdade de Ci\^encias da Universidade de Lisboa,
Campo Grande, Edif\'icio C6, 1749-016 Lisboa, Portugal {\rm and} Grupo de F\'isica Matem\'atica, Instituto Superior T\'ecnico, Av.\ Rovisco Pais, 1049-001 Lisboa, Portugal}
\email{jbkennedy@ciencias.ulisboa.pt}
\address{Departamento de Matem\'atica, Faculdade de Ci\^encias da Universidade de Lisboa,
Campo Grande, Edif\'icio C6, 1749-016 Lisboa, Portugal}
\email{jpgribeiro@ciencias.ulisboa.pt}

\thanks{The authors would like to thank Matthias Hofmann for helpful discussions about optimal Cheeger partitions, as well as the anonymous referee for a very careful and thorough reading of the first draft of this manuscript, and a number of constructive comments and suggestions. Both authors were supported by the Funda\c{c}\~ao para a Ci\^encia e a Tecnologia, Portugal, within the scope of the project NoDES: Nonlinear Dispersive and Elliptic Systems - new horizons in regularity, dynamics and asymptotic analysis, reference PTDC/MAT-PUR/1788/2020, and via grant UIDB/00208/2020.}

\begin{document}

\begin{abstract}
We study partition problems based on two ostensibly different kinds of energy functionals defined on $k$-partitions of metric graphs: Cheeger-type functionals whose minimisers are the $k$-Cheeger cuts of the graph, and the corresponding values are the $k$-Cheeger constants of the graph; and functionals built using the first eigenvalue of the Laplacian with positive, i.e. absorbing, Robin (delta) vertex conditions at the boundary of the partition elements. We prove existence of minimising $k$-partitions, $k \geq 2$, for both these functionals. We also show that, for each $k \geq 2$, as the Robin parameter $\alpha \to 0$, up to a renormalisation the spectral minimal Robin energy converges to the $k$-Cheeger constant. Moreover, up to a subsequence, the Robin spectral minimal $k$-partitions converge in a natural sense to a $k$-Cheeger cut of the graph. Finally, we show that as $\alpha \to \infty$ there is convergence in a similar sense to the corresponding Dirichlet minimal energy and partitions.

It is strongly expected that similar results hold on general (smooth, bounded) Euclidean domains and manifolds.
\end{abstract}

\maketitle
%\tableofcontents

\section{Introduction}
\label{sec:intro}

We are interested in certain kinds of geometric and analytic ways of partitioning an object such as a domain or a graph, and the relation between them. Given a domain $\Omega \subset \R^d$ and $k \geq 2$, consider the problem of finding a partition of $\Omega$ into $k$ nonempty disjoint subdomains $\Omega_1, \ldots, \Omega_k$ in such a way that
\begin{equation}
\label{eq:cheeger-2}
    \max \left\{\frac{|\partial\Omega_1|}{|\Omega_1|},\ldots,\frac{|\partial\Omega_k|}{|\Omega_k|} \right\}
\end{equation}
is minimised, that is, such that the maximum of the respective \emph{isoperimetric ratios} of the subdomains is as small as possible. Under reasonable conditions, when $k=2$ this minimal value corresponds to the \emph{Cheeger constant} of $\Omega$ as introduced in \cite{Ch70} for manifolds. We will refer to the problem of finding a $k$-partition of $\Omega$ which minimises \eqref{eq:cheeger-2}, as the \emph{Cheeger problem} (or \emph{$k$-Cheeger problem}) for $\Omega$, the corresponding minimal value as its $k$-Cheeger constant, and any corresponding minimal partition as a $k$-Cheeger cut. Such problems have been studied exhaustively in many different contexts, including domains and (discrete) graphs, both for their links to Laplacian eigenvalues and for their intrinsic geometric interpretation and applications; see, e.g., \cite{BoPa18,BrTaChSz14,Ch05,Pa11}, as well as Section~\ref{sec:cheeger} below.

The Cheeger problem is in general analytically and computationally difficult; the associated functionals are defined on BV spaces rather than $L^p$ spaces, and it is well known that the problem on discrete graphs, where it finds applications in clustering, is NP hard (e.g., \cite{BrTaChSz14}).

One approach is to approximate the Cheeger cut via the (in some senses singular) limit, as $p\to 1$, of partitions based on eigenfunctions of the $p$-Laplacian with Dirichlet boundary conditions (e.g., \cite[Section~7]{Pa11}; see also \cite{BoBuFr20} for another approach). Here, however, we will propose and study another class of Laplacian-based \emph{spectral minimal partition} problems which, as we will show, also give access to the Cheeger problem, and which we believe to be interesting in their own right. For a number of reasons, as we will explain below, in the current work we will study these problems in the context of \emph{metric graphs} rather than domains or manifolds; however, we will now explain the basic ideas in the domain case, where we expect the results to continue to hold.

Given a subdomain $D \subset \Omega \subset \R^d$, and a number $\alpha > 0$, consider the first eigenvalue $\robinlambda[\alpha] (D)$ of the Laplacian on $D$ with Robin boundary conditions, defined as the smallest value of $\lambda$ for which the problem
\begin{displaymath}
\begin{aligned}
    -\Delta u &= \lambda u \qquad &&\text{in } D,\\
    \frac{\partial u}{\partial\nu} + \alpha u &= 0 \qquad &&\text{on } \partial D
\end{aligned}
\end{displaymath}
admits a nontrivial solution $u$. (Here $\partial D$ is, in principle, the complete topological boundary of $D$, although one could theoretically consider a variant with Neumann conditions on $\partial D \cap \partial \Omega$, and $\frac{\partial u}{\partial\nu}$ is the outward-pointing normal derivative to $u$ on $\partial D$; for simplicity and ease of exposition we are ignoring certain questions of regularity of $D$.) It is known that $\robinlambda[\alpha] (D) > 0$ as long as $\alpha > 0$; moreover, there is a well-known formula (see the references in Section~\ref{sec:robin-dependence-alpha}) which asserts that
\begin{equation}
\label{eq:robin-expansion-at-zero-domains}
    \robinlambda[\alpha] (D) = \alpha \frac{|\partial D|}{|D|} + o(\alpha)
\end{equation}
as $\alpha \to 0$. If, given a partition $\partition = (\Omega_1, \ldots, \Omega_k)$ of $\Omega$ into $k$ nonempty pairwise disjoint subdomains, we look at
\begin{displaymath}
    \Lambda^\alpha := \Lambda^\alpha (\partition) := \max \left\{\robinlambda[\alpha] (\Omega_1), \ldots, \robinlambda[\alpha] (\Omega_k) \right\}, 
\end{displaymath}
then it follows from \eqref{eq:robin-expansion-at-zero-domains} that $\frac{1}{\alpha} \Lambda^\alpha (\partition)$ converges exactly to the quantity \eqref{eq:cheeger-2} as $\alpha \to 0$. It might then be \emph{hoped} (and we will show in the context of metric graphs, for general $k$-partitions) that this convergence for fixed $\partition$ carries over to the infimum of $\frac{\Lambda^\alpha}{\alpha}$ among all partitions of $\Omega$, that is, that this infimum converges to the Cheeger constant as $\alpha \to 0$.

The problem of minimising a functional of the form $\Lambda^\alpha$ among all $k$-partitions of $\Omega$ is that of finding a so-called \emph{spectral minimal partition}, another type of partition problem which has been well studied in the last couple of decades, albeit almost exclusively for Laplacians with Dirichlet rather than Robin boundary conditions, see \cite{BNH17}. To date spectral minimal partitions based on Robin Laplacian eigenvalues have not been studied, with one notable exception \cite{BuFrGi18}, where the authors prove existence of a partition minimising the sum of the Robin eigenvalues $\sum_{i=1}^k \robinlambda[\alpha] (\Omega_i)$ rather than their maximum. Actually, for these sums, some of the same authors also used the same basic principle, the formula \eqref{eq:robin-expansion-at-zero-domains}, to prove a version of the so-called \emph{honeycomb conjecture} for the Robin Laplacian in \cite{BuFr19ccm}, see also \cite{BuFr19adv} (although their mechanic was somewhat different, in particular looking at the regime $k \to \infty$ for fixed $\alpha$, rather than $\alpha \to 0$ for fixed $k$).

As mentioned, we will study these problems in the specific setting of (compact) metric graphs (roughly, one may imagine a collection of finitely many bounded intervals glued together in a particular way at their endpoints; for a somewhat more formal introduction and references see Section~\ref{sec:cheeger}). For such a first attempt to link Cheeger and Robin partitions, graphs are a useful sandbox with major advantages over domains and manifolds: on the one hand, they still display complex (analytic and spectral) behaviour analogous to domains and manifolds; on the other, being locally one-dimensional objects they allow one to suppress or simplify many technical issues. For example, the boundary of a connected subgraph, or, equivalently, any Cheeger cut of the metric graph, is a finite set.

This has various consequences: first, existence results become greatly simplified, since only finitely many topologically distinct $k$-partitions of a given graph are possible; this induces a division of the set of all $k$-partitions into a finite number of equivalence classes, whereby within an equivalence class the number of boundary points on each edge is given. In particular, within an equivalence class one only needs to specify edge lengths of the partition elements to determine the partition completely, resulting in an essentially compact space. These ideas were exploited in \cite{KeKuLeMu21} to construct quite a general existence theory, which will be exploited here. Moreover, once existence is established, regularity of the optimisers is a non-issue.

Second, it is much easier to talk about convergence not just of the partition energies but of the partitions themselves, as there is a natural notion of Hausdorff convergence of partition subgraphs within each equivalence class.

But the essentially one-dimensional nature of the partition elements will allow us to prove certain monotonicity and continuity results for the Robin eigenvalues with respect to domain inclusion; this, in turn, will allow a kind of uniform control of the error term in \eqref{eq:robin-derivative-general} which will be critical for proving convergence of the normalised partition energies and (up to subsequences) the partitions themselves as $\alpha \to 0$. We will also be able to show, using similar ideas, that both the minimal Robin partition energies and (up to subsequences) the minimal Robin partitions themselves converge to their Dirichlet Laplacian counterparts as $\alpha \to \infty$.

Thus, at least at the level of the minimal partitions themselves, the Robin partitions interpolate between Dirichlet spectral minimal partitions and $k$-Cheeger cuts as $\alpha$ ranges from $\infty$ down to $0$.

It is to be strongly expected that analogous results hold for Euclidean domains; however, due to the technical and conceptual complexities arising from this case it will be deferred to future work.

This paper is organised as follows. In Section~\ref{sec:cheeger}, after a description of our notation and definitions around graphs and partitions of them, we will formally introduce general Cheeger cuts, Cheeger constants and the related Cheeger functionals on metric graphs, as well as formulate and prove the existence of a minimising partition for the Cheeger functionals (that is, the existence of a $k$-Cheeger cut for $k \geq 2$) in Theorem~\ref{thm:cheeger}. In this section we also give a self-contained description of partition convergence, see Section~\ref{sec:convergence-description}.

A particularity of metric graphs is, of course, what happens at the vertices; it turns out that if a vertex of higher degree is a boundary point, then the ``size'' of the boundary at that point is not canonically defined; this leads us to introduce the notion of the \emph{effective degree} of a vertex within a given subgraph, see Definition~\ref{def:boundary} and Section~\ref{sec:example}, as what we consider to be the right notion. This also coincides with the \emph{perimeter} of a metric subgraph as introduced recently in \cite{Ma23,Ma22}, as we will show in Appendix~\ref{sec:appendix}; see in particular Proposition~\ref{prop:perimeter-characterisation}.

In Section~\ref{sec:robin-smp} we will formally introduce the Robin and Dirichlet Laplacian eigenvalues and related functionals, and then state our main theorems: Theorem~\ref{thm:robin-existence} establishes existence of Robin spectral minimal partitions for every fixed $\alpha>0$, and some basic properties of their behaviour with respect to the boundary parameter $\alpha>0$; Theorem~\ref{thm:robin-limit} establishes convergence of the normalised spectral minimal energies and the corresponding partitions up to subsequences, both to their Cheeger counterparts as $\alpha \to 0$ and to their Dirichlet counterparts as $\alpha \to \infty$.

Section~\ref{sec:robin} is devoted to properties of the first eigenvalue of the Robin Laplacian on subgraphs of a given underlying metric graph, both in function of the parameter $\alpha$ and the edge lengths. These results are elementary, but to keep things self-contained and to ensure that all auxiliary results are available in exactly the form we need, we include the proofs.

The proof of Theorem~\ref{thm:robin-existence} is then given in Section~\ref{sec:robin-existence}, while the proof of Theorem~\ref{thm:robin-limit} is the subject of Section~\ref{sec:robin-cheeger}. In Appendix~\ref{sec:cheeger-alt} we collect a few remarks about some of the definitions we have used (in particular, regarding Cheeger partitions and the size of the boundary of a subgraph) and compare them with other possibilities.

\section{Cheeger cuts and partitions}
\label{sec:cheeger}

We start with some basics on subgraphs and partitions of graphs, necessary to introduce our Cheeger-type functionals and minimisation problem. Throughout the paper, we will always take $\Gamma = (\VertexSet,\EdgeSet)$ to be a fixed compact, connected metric graph with finite vertex set $\VertexSet = \VertexSet (\Gamma)$ and finite edge set $\EdgeSet = \EdgeSet (\Gamma)$ consisting of non-oriented edges, where multiple parallel edges and loops are allowed. In particular, every edge $e \in \EdgeSet$ of $\Gamma$ may be identified with an interval of finite length, and $\Gamma$ will be canonically identified with a compact metric space. Moreover, as is standard, whenever necessary we will assume an arbitrary point $x \in \Gamma$ to be a (possibly dummy) vertex. See, e.g., \cite{BeKu13,Mu19} for more details.

\subsection{Subgraphs and their boundary}

For technical reasons, we always assume that our subgraphs $\Omega \subset \Gamma$ are closed, and without loss of generality $\partial\Omega \subset \VertexSet (\Gamma)$. It will also be necessary to distinguish sub\emph{graphs} $\Omega$ from sub\emph{sets} of $\Gamma$: in the former case, we allow different connectivity relations at the boundary $\partial\Omega$ in $\Omega$ as compared with what happens in $\Gamma$. This will be necessary to handle the natural notion of subgraph and partition convergence arising from taking limits, as will be discussed in Sections~\ref{sec:example} and~\ref{sec:convergence-description}, and in particular guarantee lower semicontinuity of the Robin-type functionals with respect to these limits.

To illustrate what we mean about boundary connectivity, consider the graph $\Gamma$ depicted in Figure~\ref{fig:mutant-tadpole}.
\begin{figure}[H]
\begin{tikzpicture}%[scale=1.2]
%\draw[thick] (2,0) -- (4,0);
\draw[thick,bend left=90]  (2,0) edge (4,0);
\draw[thick,bend right=90]  (2,0) edge (4,0);
\draw[thick,bend left=25] (2,0) edge (4,0);
\draw[thick, bend right=25] (2,0) edge (4,0);
\draw[fill] (2,0) circle (1.5pt);
\draw[fill] (4,0) circle (1.5pt);
\node at (1.8,0) [anchor=east] {$v$};
\draw[thick] (2,0) -- (-0.25,1);
\draw[thick] (2,0) -- (-0.25,-1);
\draw[fill] (-0.25,1) circle (1.5pt);
\draw[fill] (-0.25,-1) circle (1.5pt);
\node at (0.875,0.5) [anchor=south] {$e_1$};
\node at (0.875,-0.5) [anchor=north] {$e_2$};
\end{tikzpicture}
\caption{The graph $\Gamma$, not to scale, which we will use as a running example. We will assume $|e_1| = |e_2| = 1$, while the other edges have length $\frac{1}{2}$ each.}\label{fig:mutant-tadpole}
\end{figure}
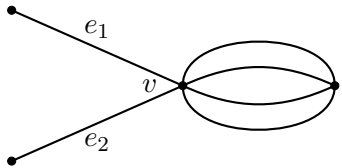
In Figure~\ref{fig:subgraph-topology}, we see two different subgraphs, $\Omega_1$ and $\Omega_2$ associated with the same subset of $\Gamma$, namely the union of the four parallel edges (a ``4-pumpkin'').
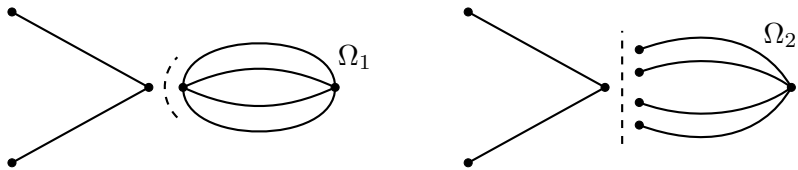
\begin{figure}[H]
\begin{tikzpicture}%[scale=1.2]
%\draw[thick] (2,0) -- (4,0);
\draw[thick,bend left=90]  (2,0) edge (4,0);
\draw[thick,bend right=90]  (2,0) edge (4,0);
\draw[thick,bend left=25] (2,0) edge (4,0);
\draw[thick, bend right=25] (2,0) edge (4,0);
\draw[fill] (2,0) circle (1.5pt);
\draw[fill] (4,0) circle (1.5pt);
%\node at (1.8,0) [anchor=east] {$v$};
\draw[thick] (1.55,0) -- (-0.25,1);
\draw[thick] (1.55,0) -- (-0.25,-1);
\draw[fill] (1.55,0) circle (1.5pt);
\draw[fill] (-0.25,-1) circle (1.5pt);
\draw[fill] (-0.25,1) circle (1.5pt);
\draw[thick,dashed,bend left=45] (1.925,-0.4) edge (1.925,0.4);
%\node at (0.875,0.55) [anchor=south] {$\Omega_{1}$};
%\node at (0.875,-0.55) [anchor=north] {$\Omega_{2}$};
\node at (3.9,0.4) [anchor=west] {$\Omega_1$};
\draw[thick] (7.55,0) -- (5.75,1);
\draw[thick] (7.55,-0) -- (5.75,-1);
\draw[thick,dashed] (7.775,-0.75) -- (7.775,0.75);
\draw[fill] (7.55,0) circle (1.5pt);
\draw[fill] (5.75,-1) circle (1.5pt);
\draw[fill] (5.75,1) circle (1.5pt);
%\node at (6.875,0.55) [anchor=south] {$\Omega_{1}$};
%\node at (6.875,-0.55) [anchor=north] {$\Omega_{2}$};
\draw[thick] (8,0.5) .. controls (8.8,0.8) and (9.6,0.7) .. (10,0);
\draw[thick] (8,0.2) .. controls (8.8,0.5) and (9.6,0.3) .. (10,0);
\draw[thick] (8,-0.2) .. controls (8.8,-0.5) and (9.6,-0.3) .. (10,0);
\draw[thick] (8,-0.5) .. controls (8.8,-0.8) and (9.6,-0.7) .. (10,0);
\draw[fill] (8,0.5) circle (1.5pt);
\draw[fill] (8,0.2) circle (1.5pt);
\draw[fill] (8,-0.2) circle (1.5pt);
\draw[fill] (8,-0.5) circle (1.5pt);
\draw[fill] (10,0) circle (1.5pt);
\node at (9.5,0.4) [anchor=south west] {$\Omega_2$};
\end{tikzpicture}
\caption{Two different subgraphs of $\Gamma$ associated with the same subset. This example will be explored further in Section~\ref{sec:example}.}
\label{fig:subgraph-topology}
\end{figure}
This corresponds to the notion of cutting through a vertex (cf.\ \cite[Section~3.1]{BeKeKuMu19}). That is, given a subgraph $\Omega \subset \Gamma$ and a vertex $v \in \VertexSet (\Gamma) \cap \partial\Omega$, regarded as a point in $\Gamma$, cutting through $v$ leads to a natural multi-valued mapping from $v$ onto a set of points (without loss of generality vertices) in $\Omega$, the images of $v$ under the cut.

\begin{definition}
\label{def:subgraph-boundary-vertex}
Given a subgraph $\Omega \subset \Gamma$ and a vertex $v \in \VertexSet (\Gamma) \cap \partial\Omega$, we refer to the image vertices of $v$ in $\Omega$ under the cut, as the \emph{descendants} of $v$ in $\Omega$.
\end{definition}

In the examples in Figure~\ref{fig:subgraph-topology}, $v$ has one descendant in $\Omega_1$, and four in $\Omega_2$.

In practice, in an attempt to simplify notation and language, whenever it does not lead to ambiguities we will not distinguish notationally between a subgraph $\Omega$ of $\Gamma$ and the subset of $\Gamma$ from which it is extracted (called the \emph{support} of $\Omega$ in \cite{KeKuLeMu21}). Unless otherwise specified, $\VertexSet (\Gamma) \cap \partial\Omega$ will denote the topological boundary of $\Omega$ as a subset of $\Gamma$, and $\partial\Omega$ will denote the set of all descendant vertices of $\VertexSet (\Gamma) \cap \partial\Omega$, which is a set of points in the subgraph $\Omega$.

We are interested in measuring the size of such boundary sets $\partial\Omega$. An important subtlety arises when a vertex $v$ of $\Gamma$ of degree at least three lies on the boundary of $\Omega$: 

\begin{definition}
\label{def:boundary}
Let $\Omega$ be a subgraph of $\Gamma$ and let $v \in \Gamma$ be an arbitrary point of $\partial\Omega$ (without loss of generality a vertex, possibly of degree two). Suppose it has descendants $v_1,\ldots,v_m$, $m\geq 1$, in $\Omega$.
\begin{enumerate}
\item[{\bf (a)}] The \emph{degree of $v_i$ in $\Omega$}, $\deg_\Omega(v_i)$, is the number of connected components of
\begin{displaymath}
    B_\varepsilon (v_i) \setminus \{v_i\}
\end{displaymath}
for any $\varepsilon>0$ sufficiently small (where $B_\varepsilon (v_i)$ is the open ball in the subgraph $\Omega$).
\item[{\bf (b)}] The \emph{effective degree} of $v_i$ in $\Omega$ will be defined by
\begin{displaymath}
    \EffDeg_\Omega (v_i) := \min \{\deg_\Omega(v_i), \deg_\Gamma(v)
    - \deg_\Omega(v_i)\}.
\end{displaymath}
\item[{\bf (c)}] The \emph{size of the boundary} $\partial \Omega$ will be given by
\begin{displaymath}
    |\partial\Omega| := \sum_{v_i \in \partial\Omega} \EffDeg_\Omega (v_i).
\end{displaymath}
\end{enumerate}
\end{definition}

While this definition may appear strange at first glance, it essentially gives the largest possible size of the boundary of a subgraph $\Omega$ which is always lower semicontinuous with respect to shifting the boundary vertices of $\Omega$ within $\Gamma$; we will give a prototypical example of this in Section~\ref{sec:example}. See also the discussion of possible alternatives in Appendix~\ref{app:alternative-boundary}.

In the important special case where the connectivity of $\Omega$ at $\partial\Omega$ is unchanged from $\Gamma$, that is, where each vertex $v \in \VertexSet (\Gamma) \cap \partial\Omega$ has only one descendant, it also means that $|\partial\Omega|$ coincides with the \emph{perimeter} of $\Omega$, $\perim (\Omega)$, as introduced recently in \cite{Ma23,Ma22}, see also \eqref{eq:perimeter}; we will prove this in Appendix~\ref{sec:appendix}, see Proposition~\ref{prop:perimeter-characterisation}. In this case we will also not distinguish between $v \in \VertexSet (\Gamma) \cap \partial\Omega$ and its unique descendant.

In the example in Figure~\ref{fig:subgraph-topology}, the effective degree in $\Omega_1$ of the central vertex is $2 = \min \{4, 2\}$, meaning $|\partial\Omega_1|=2$; on the right-hand side, each of the four image vertices of $v$ has effective degree $1 = \min \{1, 5\}$, so $|\partial\Omega_2|=4$. 

\subsection{Partitions and Cheeger functionals}

We can now define precisely what we mean by a partition. For the following definition we assume without loss of generality that all boundary points of the subgraphs are (possibly dummy) vertices in $\Gamma$.

\begin{definition}
\label{def:partition}
Given $\Gamma$, we call a collection $\partition = (\Omega_1, \ldots, \Omega_k)$ of $k \geq 1$ connected subgraphs $\Omega_1,\ldots,\Omega_k$ of $\Gamma$ with nonempty interiors a \emph{$k$-partition} if, for all $i \neq j$, $\Omega_i \cap \Omega_j \subset \VertexSet (\Gamma) \cap \partial\Omega_i \cap \partial\Omega_j$ is a finite set in $\Gamma$.
\end{definition}

Following \cite{KeKuLeMu21}, a partition $\partition = (\Omega_1,\ldots,\Omega_k)$ is said to be \emph{exhaustive} (or, following \cite{BNH17}, \emph{strong}; one also sometimes speaks of \emph{subpartitions}) if $\bigcup_{i=1}^k \Omega_i = \Gamma$; in general we will \emph{not} assume this property.

We can now introduce our first main object of study, the Cheeger functional of a partition of $\Gamma$.

\begin{definition}
\label{def:cheeger}
Let $\Gamma$ be a connected, compact metric graph.
\begin{enumerate}
\item[{\bf (a)}] We define the \emph{Cheeger functional} $\energyc : \partition \mapsto \R$ on the set of all partitions $\partition$ of $\Gamma$ by
\begin{equation}
\label{eq:energyc}
    \energyc (\partition) = \max_{i=1,\ldots,k} \frac{|\partial \Omega_i|}{|\Omega_i|},
\end{equation}
where $\partition = (\Omega_1,\ldots,\Omega_k)$. By way of analogy with integral functionals, we will call the value that $\energyc$ takes on a given partition $\partition$, $\energyc (\partition)$, the \emph{(Cheeger) energy} of $\partition$.
\item[{\bf (b)}] The \emph{$k$-th Cheeger constant} of $\Gamma$, $k \geq 2$, is given by
\begin{equation}
\label{eq:cheegerk}
    \cheeger[k] (\Gamma) = \inf \{\energyc (\partitionk[k]): \partitionk[k] \text{ is a $k$-partition of $\Gamma$} \}.
\end{equation}
\item[{\bf (c)}] Any $k$-partition of $\Gamma$ on which the infimum in \eqref{eq:cheegerk} is attained will be called a \emph{$k$-Cheeger cut} of $\Gamma$.
\end{enumerate}
\end{definition}

As intimated above, our definition of $|\partial\Omega_i|$ ensures that $\energyc$ is lower semicontinuous in an appropriate sense (see Section~\ref{sec:cheeger-proof}); this, in turn, leads to a fundamental existence result.

\begin{theorem}
\label{thm:cheeger}
Let $\Gamma$ be a connected, compact graph. For all $k\geq 2$ there exists a $k$-partition $\partitionk[k]^C$ of $\Gamma$ such that
\begin{displaymath}
    \cheeger[k] (\Gamma) = \energyc (\partitionk[k]^C).
\end{displaymath}
\end{theorem}

(Note that when $k=1$ all the above definitions make sense, and the theorem continues to hold, but in this case the \emph{first Cheeger constant}, $\cheeger[1] (\Gamma)$, would be equal to $0$, the minimum being attained by the trivial partition $\partition = (\Gamma)$.)

There is another way one can define Cheeger partitions, which is perhaps more common on domains. This definition could also be used here; however, since it is not essential for our main results, they will be discussed in Appendix~\ref{app:cheeger-alt}.

There are also other possible ways to define $|\partial\Omega|$ which would still lead to Theorem~\ref{thm:cheeger} (albeit with a lower value of $\cheeger[k] (\Gamma)$); see Appendix~\ref{app:alternative-boundary}. Note, however, that the na\"ive choice $|\partial\Omega| = \sum_{v \in \partial\Omega} \deg_\Omega (v)$ will \emph{not} work; for this definition a minimising partition may not exist. See Section~\ref{sec:example} for an example. We will prove Theorem~\ref{thm:cheeger} using the framework developed in \cite{KeKuLeMu21} afterwards, in Section~\ref{sec:cheeger-proof}, after we have discussed our notion of partition convergence in Section~\ref{sec:convergence-description}.

\begin{remark}
\label{rem:cheeger-p-type}
The proof of Theorem~\ref{thm:cheeger} can also be trivially adapted to show that the minimisation problem for $p$-norm-type functionals of the form
\begin{displaymath}
    \energyc_p (\partition) = \left[\sum_{i=1}^k \left(
    \frac{|\partial\Omega_i|}{|\Omega_i|}\right)^p\right]^{1/p}
\end{displaymath}
$1\leq p < \infty$, also admits a solution; since we are only interested in the $\infty$-norm case, which corresponds to the classical Cheeger problem, we will not go into details. See also Remark~\ref{rem:robin-p-type}.
\end{remark}

\subsection{A basic example}
\label{sec:example}

To illustrate why $\sum_{v \in \partial\Omega} \deg_\Omega (v)$ will not work as a definition of $|\partial\Omega|$ to ensure the existence of a $k$-Cheeger cut, we will consider the optimal $3$-cut of the graph $\Gamma$ depicted in Figure~\ref{fig:mutant-tadpole}, consisting of two pendant edges of length $1$ each, joined at a common vertex $v$ to a \emph{$4$-pumpkin graph}, each of whose edges has length $\frac{1}{2}$.

If in a first step we avoid cutting through the vertex $v$ and only allow cuts through edges, we can find a sequence of $3$-partitions $\partition_n = (\Omega_{i,n})_{i=1}^3$ for which $\energyc (\partition_n) \to 1^+$:
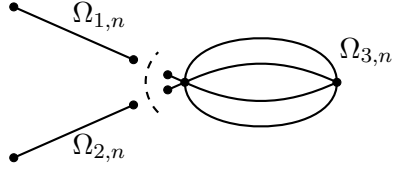
\begin{figure}[H]
\begin{tikzpicture}%[scale=1.2]
%\draw[thick] (2,0) -- (4,0);
\draw[thick,bend left=90]  (2,0) edge (4,0);
\draw[thick,bend right=90]  (2,0) edge (4,0);
\draw[thick,bend left=25] (2,0) edge (4,0);
\draw[thick, bend right=25] (2,0) edge (4,0);
\draw[fill] (2,0) circle (1.5pt);
\draw[fill] (4,0) circle (1.5pt);
%\node at (1.8,0) [anchor=east] {$v$};
\draw[thick] (2,0) -- (1.775,0.1);
\draw[thick] (2,0) -- (1.775,-0.1);
\draw[thick] (1.325,0.3) -- (-0.25,1);
\draw[thick] (1.325,-0.3) -- (-0.25,-1);
\draw[fill] (1.775,0.1) circle (1.5pt);
\draw[fill] (1.775,-0.1) circle (1.5pt);
\draw[fill] (1.325,0.3) circle (1.5pt);
\draw[fill] (1.325,-0.3) circle (1.5pt);
\draw[fill] (-0.25,-1) circle (1.5pt);
\draw[fill] (-0.25,1) circle (1.5pt);
\draw[thick,dashed,bend left=45] (1.65,-0.4) edge (1.65,0.4);
\node at (0.875,0.55) [anchor=south] {$\Omega_{1,n}$};
\node at (0.875,-0.55) [anchor=north] {$\Omega_{2,n}$};
\node at (3.9,0.4) [anchor=west] {$\Omega_{3,n}$};
%\node at (2,0) [anchor=north] {$v$};
\end{tikzpicture}
\caption{A $3$-partition $\partition_n$ of $\Gamma$. As $n \to \infty$ the cut points should move towards $v$.}\label{fig:mutant-dissection}
\end{figure}

Indeed, with the above setup, clearly
\begin{displaymath}
    1 < \frac{|\partial \Omega_{1,n}|}{|\Omega_{1,n}|}
    = \frac{|\partial \Omega_{2,n}|}{|\Omega_{2,n}|} \longrightarrow 1,
    \qquad
    1 > \frac{|\partial \Omega_{3,n}|}{|\Omega_{3,n}|} \longrightarrow
    \frac{2}{4\cdot \tfrac{1}{2}} = 1.
\end{displaymath}
A simple but tedious examination of all possible different cases, whose details we will omit, shows that no partition which does not involve cutting through the vertex $v$ can have Cheeger energy $\leq 1$; thus $\cheeger[3] (\Gamma) = 1$. Now the limit partition $\partition = (\Omega_i)_{i=1}^3$, shown in Figure~\ref{fig:mutant-3-cut} (left), is the only partition which can possibly have energy $1$. But its \emph{actual} energy depends on how $|\partial\Omega_3|$ is defined: any value \emph{greater} than $2$ would mean that $\energyc (\partition) > 1$ and thus no minimiser exists.
\begin{figure}[H]
\begin{tikzpicture}[scale=0.8]%[scale=1.2]
%\draw[thick] (2,0) -- (4,0);
\draw[thick,bend left=90]  (2,0) edge (4,0);
\draw[thick,bend right=90]  (2,0) edge (4,0);
\draw[thick,bend left=25] (2,0) edge (4,0);
\draw[thick, bend right=25] (2,0) edge (4,0);
\draw[fill] (2,0) circle (1.5pt);
\draw[fill] (4,0) circle (1.5pt);
%\node at (1.8,0) [anchor=east] {$v$};
\draw[thick] (1.55,0.2) -- (-0.25,1);
\draw[thick] (1.55,-0.2) -- (-0.25,-1);
\draw[fill] (1.55,0.2) circle (1.5pt);
\draw[fill] (1.55,-0.2) circle (1.5pt);
\draw[fill] (-0.25,-1) circle (1.5pt);
\draw[fill] (-0.25,1) circle (1.5pt);
\draw[thick,dashed,bend left=45] (1.925,-0.4) edge (1.925,0.4);
\node at (0.875,0.55) [anchor=south] {$\Omega_{1}$};
\node at (0.875,-0.55) [anchor=north] {$\Omega_{2}$};
\node at (3.9,0.4) [anchor=west] {$\Omega_{3}$};
\draw[thick] (7.55,0.2) -- (5.75,1);
\draw[thick] (7.55,-0.2) -- (5.75,-1);
\draw[thick,dashed] (7.775,-0.75) -- (7.775,0.75);
\draw[fill] (7.55,0.2) circle (1.5pt);
\draw[fill] (7.55,-0.2) circle (1.5pt);
\draw[fill] (5.75,-1) circle (1.5pt);
\draw[fill] (5.75,1) circle (1.5pt);
\node at (6.875,0.55) [anchor=south] {$\Omega_{1}$};
\node at (6.875,-0.55) [anchor=north] {$\Omega_{2}$};
\draw[thick] (8,0.5) .. controls (8.8,0.8) and (9.6,0.7) .. (10,0);
\draw[thick] (8,0.2) .. controls (8.8,0.5) and (9.6,0.3) .. (10,0);
\draw[thick] (8,-0.2) .. controls (8.8,-0.5) and (9.6,-0.3) .. (10,0);
\draw[thick] (8,-0.5) .. controls (8.8,-0.8) and (9.6,-0.7) .. (10,0);
\draw[fill] (8,0.5) circle (1.5pt);
\draw[fill] (8,0.2) circle (1.5pt);
\draw[fill] (8,-0.2) circle (1.5pt);
\draw[fill] (8,-0.5) circle (1.5pt);
\draw[fill] (10,0) circle (1.5pt);
\node at (9.5,0.4) [anchor=south west] {$\widetilde\Omega_{3}$};
\draw[thick] (11.75,-1) -- (13.55,0);
\draw[thick] (13.55,0) -- (13.325,0.1);
\draw[thick] (13.55,0) -- (13.85,0.3);
\draw[thick] (13.55,0) -- (13.85,0.1);
\draw[thick] (13.55,0) -- (13.85,-0.1);
\draw[thick] (13.55,0) -- (13.85,-0.3);
\draw[thick] (11.75,0.9) -- (13.1,0.2);
\draw[thick] (14.2,0.5) .. controls (15,0.8) and (15.8,0.7) .. (16.2,0);
\draw[thick] (14.2,0.2) .. controls (15,0.5) and (15.8,0.3) .. (16.2,0);
\draw[thick] (14.2,-0.2) .. controls (15,-0.5) and (15.8,-0.3) .. (16.2,0);
\draw[thick] (14.2,-0.5) .. controls (15,-0.8) and (15.8,-0.7) .. (16.2,0);
\draw[fill] (13.1,0.2) circle (1.5pt);
\draw[fill] (11.75,0.9) circle (1.5pt);
\draw[fill] (13.325,0.1) circle (1.5pt);
\draw[fill] (11.75,-1) circle (1.5pt);
\draw[fill] (13.55,0) circle (1.5pt);
\draw[fill] (13.85,0.3) circle (1.5pt);
\draw[fill] (13.85,0.1) circle (1.5pt);
\draw[fill] (13.85,-0.1) circle (1.5pt);
\draw[fill] (13.85,-0.3) circle (1.5pt);
\draw[fill] (14.2,0.5) circle (1.5pt);
\draw[fill] (14.2,0.2) circle (1.5pt);
\draw[fill] (14.2,-0.2) circle (1.5pt);
\draw[fill] (14.2,-0.5) circle (1.5pt);
\draw[fill] (16.2,0) circle (1.5pt);
\draw[thick,dashed] (14.025,-0.7) -- (14.025,0.7);
\draw[thick,dashed] (13.45,0.3125) -- (12.95,-0.0625);
\node at (12.875,0.55) [anchor=south] {$\widetilde \Omega_{1,n}$};
\node at (12.875,-0.55) [anchor=north] {$\widetilde \Omega_{2,n}$};
\node at (15.7,0.4) [anchor=south west] {$\widetilde\Omega_{3,n}$};
\end{tikzpicture}
\caption{Left: The only candidate $3$-Cheeger cut of $\Gamma$, $\partition$. Its actual Cheeger energy depends critically on how the size of the cut at $v$ is measured. Centre: A different partition built out of the same subsets, which is the natural limit of suitable partitions $\partition_n = (\widetilde\Omega_{1,n},\widetilde\Omega_{2,n},\widetilde\Omega_{3,n})$ where $\widetilde \Omega_{3,n}$ is cut from the rest of $\Gamma$ to the right of $v$; one such possibility is depicted on the right. The partition in the centre has a strictly higher Cheeger energy (namely $2$) than the partition on the left.}\label{fig:mutant-3-cut}
\end{figure}
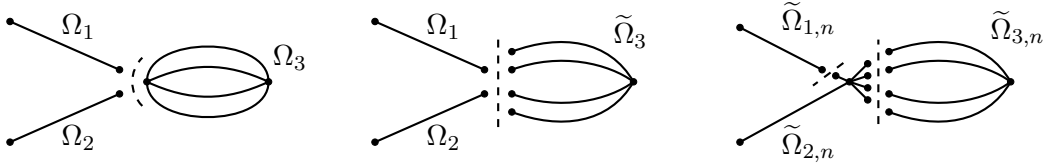

Thus, in particular, the size of a boundary vertex in $\partial\Omega$ cannot simply be its degree in $\Omega$ (here $4$). In this example, $\perim (\Omega_3) = 2$ (see Proposition~\ref{prop:perimeter-characterisation}), and it is, not coincidentally, this value which guarantees that $\energyc (\partition_n) \to \energyc (\partition)$. Indeed, $2$ is obviously the largest value for which $\energyc$ is lower semicontinuous.

In accordance with Definition~\ref{def:boundary}, $\EffDeg_{\Omega_3} (v) = \min \{4,2\} = 2$ (with $\EffDeg_{\Omega_1} (v) = \EffDeg_{\Omega_2} (v) = \min \{1,5\} = 1$); while, as mentioned, $\perim (\Omega_3) = 2$ as well. See Appendix~\ref{app:alternative-boundary} for a further discussion.

\subsection{Convergence of subgraphs and partitions}
\label{sec:convergence-description}

We briefly recall the notion of convergence of subgraphs $\Omega_n \subset \Gamma$ of a fixed graph $\Gamma$ introduced in \cite[Section~3]{KeKuLeMu21}, with respect to which our functionals will be lower semicontinuous. Here we will give slightly different (but equivalent) definitions using homeomorphisms. Given two graphs $\Gamma_1$ and $\Gamma_2$, by a homeomorphism $\Phi : \Gamma_1 \to \Gamma_2$ we mean an invertible continuous map with continuous inverse, for the respective natural topologies on $\Gamma_1$ and $\Gamma_2$. It is immediate from the continuous invertibility that if there exists such a homeomorphism $\Phi : \Gamma_1 \to \Gamma_2$, then, assuming without loss of generality that $\Gamma_1$ and $\Gamma_2$ do not have any degree-2 vertices, $\Phi$ is also a bijection from $\VertexSet (\Gamma_1)$ onto $\VertexSet (\Gamma_2)$; moreover, $\deg_{\Gamma_2} \Phi(v) = \deg_{\Gamma_1} v$ for all $v \in \VertexSet (\Gamma_1)$.

\begin{definition}
\label{def:vertex-homeo}
We call $\Phi: \Gamma \to \Gamma$ is a \emph{vertex-preserving homeomorphism} if it is a homeomorphism from $\Gamma$ to $\Gamma$ and $\Phi(v) = v$ for all $v \in \VertexSet (\Gamma)$.
\end{definition}

In this case, the continuity and invertibility of $\Phi$ also imply that the image of any given edge $e$ under $\Phi$ is exactly $e$. Intuitively, we are fixing an underlying discrete graph and varying the edge lengths. Note that, at least for some graphs $\Gamma$, there may exist non-vertex-preserving homeomorphisms: this can be the case if the underlying discrete graph has symmetries which $\Gamma$ does not. But since this is irrelevant for our purposes, we do not go into details.

\begin{definition}
\label{def:subgraph-similar}
Let $\Omega_1$ and $\Omega_2$ be two connected subgraphs of a given graph $\Gamma$, and denote the associated subsets of $\Gamma$ by $\widetilde\Omega_1$ and $\widetilde\Omega_2$, respectively. We say that $\Omega_1$ and $\Omega_2$ are \emph{similar} if, cumulatively,
\begin{enumerate}
\item[(1)] there exists a vertex-preserving homeomorphism $\Phi : \Gamma \to \Gamma$ such that $\Phi (\widetilde\Omega_1) = \widetilde\Omega_2$, and
\item[(2)] the natural mapping between the subgraphs $\Omega_1$ and $\Omega_2$ induced by $\Phi$ is also a homeomorphism.
\end{enumerate}
\end{definition}

Condition (1) says that $\Omega_2$ can be formed from $\Omega_1$ by rescaling the edges of the former; in fact, more than that, one forms $\Omega_2$ from $\Omega_1$ by, for any each edge of $\Gamma$, shifting the location of any boundary vertices of $\Omega_1$ within that edge. Condition (2) guarantees additionally that, if $v \in \VertexSet (\Gamma)$ is a boundary vertex of both $\Omega_1$ and $\Omega_2$, then there is a canonical bijection between its descendants in $\Omega_1$ and its descendants in $\Omega_2$ which preserves all connectivity relations between the two. In particular, there is a canonical bijection between $\partial\Omega_1$ and $\partial\Omega_2$; in this case we speak of similar boundary points.

Both conditions are necessary. Condition (1) is insufficient to guarantee that $\Omega_1$ and $\Omega_2$ have the same underlying discrete graph; for example, in Figure~\ref{fig:subgraph-topology}, (1) (but not (2)) is satisfied for the pair $\Omega_1$ and $\Omega_2$. Likewise, (2) by itself is also insufficient; for example, in Figure~\ref{fig:mutant-3-cut}, there is a homeomorphism from $\widetilde{\Omega}_{1,n}$ onto $\Omega_2$.

\begin{definition}
\label{def:subgraph-convergence}
Suppose the connected subgraphs $\Omega_n \subset \Gamma$ are similar for all $n \in \N$. We write $\Omega_n \to \Omega$ if the Hausdorff distance between $\Omega_n$ and $\Omega$ as subsets of $\Gamma$ tends to zero, $|\Omega_n \triangle \Omega| \to 0$, and if the graph $\Omega$ has the same connectivity as $\Omega_n$, that is, for all $x,y \in \Omega$ there are $x_n,y_n \in \Omega_n$ such that $x_n \to x$ and $y_n \to y$ in $\Gamma$, and $\dist_{\Omega_n} (x_n,y_n) \to \dist_\Omega (x,y)$.
\end{definition}

In this case, each sequence of similar boundary points $v_n \in \partial\Omega_n$ must be a Cauchy sequence converging to some $v \in \Gamma$ which will be a boundary point of $\Omega$. The connectivity requirement (2) in the definition means that if two such sequences $v_{n,1}, v_{n,2}$ converge to the same vertex $v \not\in \partial\Omega_n$ along different edges $e_{1} \ni v_{n,1}$, $e_{2} \ni v_{n,2}$, then $e_1$ and $e_2$ are not considered to be adjacent in $\Omega$ at $v$.
\begin{figure}[ht]
\begin{center}
\begin{tikzpicture}%[scale=1.2]
\draw[fill] (2,0) circle (1.5pt);
%\draw[fill] (4,0) circle (1.5pt);
\draw[thick] (2,0) -- (3,0);
\draw[thick,dashed] (3,0) -- (3.5,0);
\node at (2,0) [anchor=north] {$v$};
\draw[thick] (2,0) -- (1.5,0.2222);
\draw[thick] (2,0) -- (1.5,-0.2222);
\draw[fill,blue] (-0.25,1) circle (1.5pt);
\draw[fill,blue] (-0.25,-1) circle (1.5pt);
\node at (0.5,0.667) [anchor=south] {$e_1$};
\node at (0.5,-0.667) [anchor=north] {$e_2$};
\draw[thick,blue] (-0.25,1) -- (-0.25,-1);
\draw[thick,blue] (-0.25,1) -- (-0.55,1);
\draw[thick,blue,dashed] (-0.55,1) -- (-1.05,1);
\draw[thick,blue] (-0.25,-1) -- (-0.55,-1);
\draw[thick,blue,dashed] (-0.55,-1) -- (-1.05,-1);
\draw[thick,blue] (-0.25,-1) -- (-0.45,-1.2);
\draw[thick,blue,dashed] (-0.45,-1.2) -- (-0.85,-1.6);
\draw[thick,blue] (-0.25,1) -- (1.5,0.2222);
\draw[thick,blue] (-0.25,-1) -- (1.5,-0.2222);
\draw[fill,blue] (1.5,0.2222) circle (1.5pt);
\draw[fill,blue] (1.5,-0.2222) circle (1.5pt);
\node at (1.5,0.2222) [anchor=south] {$v_{n,1}$};
\node at (1.5,-0.2222) [anchor=north] {$v_{n,2}$};
\node at (-0.5,0) [anchor=east] {$\Omega_n$};
\draw[-{Stealth[scale=0.5,angle'=60]},line width=2.5pt] (4.25,0) -- (5.25,0);
\draw[thick,blue] (6.75,1) -- (6.75,-1);
\draw[thick,blue] (6.75,1) -- (6.45,1);
\draw[thick,blue,dashed] (6.45,1) -- (5.95,1);
\draw[thick,blue] (6.75,-1) -- (6.45,-1);
\draw[thick,blue,dashed] (6.45,-1) -- (5.95,-1);
\draw[thick,blue] (6.75,-1) -- (6.55,-1.2);
\draw[thick,blue,dashed] (6.55,-1.2) -- (6.15,-1.6);
\draw[fill,blue] (6.75,1) circle (1.5pt);
\draw[fill,blue] (6.75,-1) circle (1.5pt);
\draw[thick,blue] (6.75,1) -- (6.75,-1);
\draw[thick,blue] (6.75,1) -- (9,1);
\draw[thick,blue] (6.75,-1) -- (9,-1);
\draw[fill,blue] (9,1) circle (1.5pt);
\draw[fill,blue] (9,-1) circle (1.5pt);
\node at (7.875,1) [anchor=south] {$e_1$};
\node at (7.875,-1) [anchor=north] {$e_2$};
\node at (9,1) [anchor=west] {$v_1$};
\node at (9,-1) [anchor=west] {$v_2$};
\node at (7.875,0) {$\Omega$};
\end{tikzpicture}
\end{center}
\caption{The subgraph $\Omega_n$ (in blue), for which $v_{n,1} \to v$ and $v_{n,2} \to v$ as $n \to \infty$, on the left. The subgraph $\Omega$ obtained in the limit is, as a graph, equivalent to the one depicted on the right, in which $v_1$ and $v_2$ are the two descendant vertices of $v$ in $\Omega$ in the sense of Definition~\ref{def:subgraph-boundary-vertex}; in particular, $e_1$ and $e_2$ are not adjacent in $\Omega$.}\label{fig:subgraph-convergence}
\end{figure}
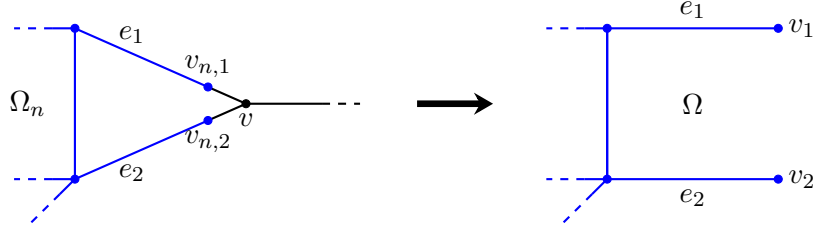

The limit graph $\Omega$ need not be similar to the $\Omega_n$ in the sense of Definition~\ref{def:subgraph-similar}; for example, in Figure~\ref{fig:mutant-3-cut}, $\widetilde{\Omega}_{2,n}$ and $\Omega_2$ are clearly not similar.

One of the key points of subgraph convergence is lower semicontinuity of the Cheeger functional. That $|\Omega_n| \to |\Omega|$ is immediate from the definition; for the boundary sets we have the following result.

\begin{lemma}
\label{lem:subgraph-convergence}
Suppose $\Omega_n \to \Omega$ in the sense of Definition~\ref{def:subgraph-convergence}. Then, for all $n \in \N$,
\begin{displaymath}
   |\partial\Omega| \leq |\partial\Omega_n|
\end{displaymath}
\end{lemma}

First note that, since the $\Omega_n$ are similar, their boundaries are always in bijection, and thus $|\partial\Omega_n|$ is necessarily constant in $n$. However, since $\Omega$ does not have to be similar to $\Omega_n$, strict inequality in Lemma~\ref{lem:subgraph-convergence} is possible. In Figure~\ref{fig:mutant-3-cut}, $\widetilde{\Omega}_{2,n} \to \Omega_2$; clearly $|\partial\widetilde{\Omega}_{2,n}| = 5$, while $|\partial\Omega_2| = 1$.

\begin{proof}
It suffices to show that, with any vertex $w \in \partial\Omega$ with $\EffDeg_\Omega (w) = d$, we can associate (i.e.\ map) uniquely a set of vertices in $\partial\Omega_n$ whose effective degrees in $\Omega_n$ sum to at least $d$. In what follows we will not distinguish between edges in $\Gamma$ and their canonical image in a subgraph $\Omega_n$ or $\Omega$.

Let $v \in \VertexSet (\Gamma) \cap \partial\Omega$ and let $\tilde v$ be an arbitrary descendant of $v$ in $\Omega$.% We suppose that $\deg_\Gamma (v) = k \geq 2$ and $\deg_\Omega (\tilde v) = d < k$, so that $\EffDeg_\Omega (\tilde v) = \min \{d, k-d\}$.

If $\deg_\Gamma (v) = 2$, that is, $v$ is in the interior of an edge, then by definition of convergence there exists a unique sequence $v_n \in \partial\Omega_n$ which, when treated as points in $\Gamma$, satisfy $v_n \to v$. In particular, $v_n \in \partial\Omega_n$ with $\EffDeg_{\Omega_n}(v_n) = \EffDeg_{\Omega}(\tilde v) = 1$. This yields the desired association.

The same argument works if $\deg_\Gamma (v) \geq 3$ but $\EffDeg_\Omega (\tilde v) = 1$, since in this case, for $n$ large enough either $\tilde v$ is a boundary vertex of $\Omega_n$ of degree $1$, or there exists a unique $v_n \in \partial\Omega_n$ in the unique edge of $\Gamma$ incident with $\tilde v$ in $\Omega$, which necessarily converges to $v$ as $n \to \infty$.

So assume that $\deg_\Gamma (v) = k \geq 3$ and $\deg_\Omega (\tilde v) = d \geq 2$, so that $\EffDeg_\Omega (\tilde v) = \min \{d, k-d\}$. Now by definition of convergence, all $d$ edges incident with $\tilde v$ in $\Omega$ are also in $\Omega_n$ for $n$ large enough, so that $\deg_{\Omega_n} (\tilde v) \geq d$ (where in a slight abuse of notation we also use write $\tilde v$ for the descendant of $v$ in $\Omega_n$ which corresponds to $\tilde v$ in $\Omega$).

Consider the remaining $d-k$ edges, call them $e_1,\ldots,e_{k-d}$, incident with $v$ in $\Gamma$ which are not incident with $\tilde v$ in $\Omega$. In each case, the assumption of the convergence of $\Omega_n$ to $\Omega$ implies that, for $n$ large enough, for each $i$ either there is a unique boundary point $v_{n_i}$ of $\partial\Omega_n$ in the interior of $e_i$, which converges to $v$ in $\Gamma$ as $n\to \infty$, or else $\interior e_i \cap \Omega_n = \emptyset$. Say there are $d_1 \leq d$ of the first kind, and so $k-d-d_1$ of the second kind.

Each of the first kind is a boundary vertex of $\Omega_n$ of effective degree $1$ which we associate with $\tilde v$ in $\Omega$; we then have $\EffDeg_{\Omega_n}(\tilde v) = \min \{d+d_1,k-d-d_1 \}$, and so the total of the effective degrees of boundary vertices in $\Omega_n$ which will converge to $\tilde v$ in $\Omega$ is
\begin{displaymath}
    d_1 + \min \{d+d_1,k-d-d_1 \} = \min \{d+d_2,k-d \} \geq \min \{d, k-d\} = \EffDeg_\Omega (\tilde v).
\end{displaymath}
This establishes the desired association for $\tilde v \in \partial\Omega$, completing the proof.
\end{proof}

The notion of similarity of subgraphs extends naturally to partitions.

\begin{definition}
\label{def:partition-convergence}
Let $\Gamma$ be a fixed connected, compact metric graph.
\begin{enumerate}
\item[\textbf{(a)}] Two $k$-partitions $\partition_1 = (\Omega_{1,1},\ldots,\Omega_{1,k})$, $\partition_2 = (\Omega_{2,1},\ldots,\Omega_{2,k})$ of $\Gamma$ are \emph{similar} if, up to a renumbering of the $\Omega_{2,j}$, $\Omega_{1,j}$ is similar to $\Omega_{2,j}$ for all $j=1,\ldots,k$.

\item[\textbf{(b)}] If $\partition_n = (\Omega_{n,1},\ldots,\Omega_{n,k})$ are similar for all $n \in \N$, then we write $\partition_n \to \partition$ for some partition $\partition = (\Omega_1,\ldots,\Omega_k)$ if each $\Omega_{n,j}$ has a limit $\Omega_j$.

\item[\textbf{(c)}] If, in (b), $|\Omega_{n,j}| \to 0$, in which case $\Omega_j$ is necessarily a finite set, then we say $\Omega_{n,j}$ is \emph{vanishing}; in this case we exclude $\Omega_j$ from the limit partition $\partition$.
\end{enumerate}
\end{definition}

Partition convergence means, in particular, that, up to a suitable ordering of the partition elements, for each $j=1,\ldots,k$ each $\Omega_{n,j}$ converges in the natural Hausdorff metric to the limit $\Omega_j$. The possibility of the ``vanishing'' case (c) means that, in general, the limit of a sequence of $k$-partitions will be an $m$-partition for some $m \leq k$. (Recall that we assume that all elements of a partition should be subgraphs with nonempty interiors, to avoid trivialities.)

As noted in \cite{KeKuLeMu21}, similarity is an equivalence relation, and on any compact graph, among $k$-partitions only finitely many equivalence classes (or \emph{configuration classes}) are possible; it follows easily that \emph{any} sequence of $k$-partitions admits a convergent subsequence (although we repeat that the limit will be an $m$-partition for some $m \leq k$).

\subsection{Proof of Theorem~\ref{thm:cheeger}}
\label{sec:cheeger-proof}

By \cite[Theorem~3.13]{KeKuLeMu21}, it suffices to show that $\energyc$ has the following two properties:
\begin{enumerate}
    \item Lower semicontinuity: $\energyc (\partition) \leq \liminf_{n \to \infty} \energyc (\partition_n)$ whenever $\partition_n \to \partition$ in the sense described in Section~\ref{sec:convergence-description}; and
    \item ``Coercivity'': $\energyc (\partition_n) \to \infty$ whenever there is a sequence of partition elements $\Omega_n \in \partition_n$ such that $|\Omega_n| \to 0$.
\end{enumerate}
(Note that showing these same two properties for suitably adjusted functionals also immediately yields existence of a minimiser for the problems described in Remark~\ref{rem:cheeger-p-type}, as well as the alternative definition of Cheeger partitions discussed in Appendix~\ref{app:cheeger-alt}.)

We note that (2) is immediate from the definition of $\energyc$, since $|\partial\Omega_n| \geq 1$ for any subgraphs $\Omega_n \subset \Gamma$ for which $|\Omega_n| \to 0$ (and thus which do not exhaust $\Gamma$).

For (1), it suffices to show that if $\Omega_n \to \Omega$ in the sense of Section~\ref{sec:convergence-description} and $\Omega$ has nonzero total length (the non-vanishing case), then
\begin{equation}
\label{eq:cheeger-lsc}
    \frac{|\partial\Omega|}{|\Omega|} \leq \liminf_{n\to\infty} \frac{|\partial\Omega_n|}{|\Omega_n|}.
\end{equation}
But this follows directly from the definition of convergence and Lemma~\ref{lem:subgraph-convergence}.

Hence \cite[Theorem~3.13]{KeKuLeMu21} is applicable, and we have existence of a minimising partition.

\section{Robin spectral minimal partitions}
\label{sec:robin-smp}

Having dealt with Cheeger partitions, we will now introduce the other main object of study of the paper, partitions built out of Robin Laplace eigenvalues. To that end we first need to introduce the eigenvalues themselves, and our notation for them; we will also do the same for their Dirichlet counterparts. We can then formulate our main theorems, Theorems~\ref{thm:robin-existence} and~\ref{thm:robin-limit}.

In keeping with our notion of effective degree (Definition~\ref{def:boundary}), and thus to be consistent with our definition of the Cheeger constant, we will define the ground state Robin energy of a domain $\Omega \subset \Gamma$, with boundary parameter (or coupling constant) $\alpha > 0$, as follows:

\begin{definition}
Let $\Omega$ be a subgraph of $\Gamma$. Given $\alpha > 0$ we define the Robin ground state energy of $\Omega$ with boundary parameter $\alpha$ by
\begin{equation}
\label{eq:robinlambda}
    \robinlambda[\alpha] (\Omega) := \inf\left\{\frac{\int_\Omega |f'(x)|^2\,\dx + \sum_{v \in \partial\Omega} \alpha \EffDeg (v) |f(v)|^2}{\int_\Omega |f(x)|^2\,\dx} : 0 \neq f \in H^1(\Omega) \right\}.
\end{equation}
\end{definition}
Here the space $H^1 (\Omega) = \{f \in C(\Omega): f|_e \in H^1(e)\, \forall e \in \EdgeSet (\Omega)\}$ is defined in the usual way (but note that continuity of functions is with respect to the topology of $\Omega$ as a closed subgraph and not a subset of $\Gamma$). We will denote the Rayleigh quotient appearing in \eqref{eq:robinlambda} by $R_\alpha [f]$, for given $0 \neq f \in H^1(\Omega)$.

It is well known (see, e.g., \cite[Section~2]{BeKeKuMu19} or \cite{BaScSo22}, and the references therein) that $\robinlambda[\alpha] (\Omega)$ is the smallest eigenvalue of the Laplacian with continuity conditions at all vertices, Neumann--Kirchhoff conditions at all interior vertices of $\Omega$, and a suitable Robin (or delta) condition 
%\begin{displaymath}
%    \sum_{\substack{e \subset \Omega \\ e \sim v}} \partial_\nu f|_e
%    +\alpha \EffDeg (v) = 0,
%\end{displaymath}
%at each vertex $v \in \partial\Omega$, where the sum is taken over all edges $e$ in $\Omega$ incident with $v$ (loops being counted twice) and $\nu$, with this orientation, is the normal derivative at $v$ pointing into $v$. Equality is attained in \eqref{eq:robinlambda} if and only if $f$ is a (classical) eigenfunction for $\robinlambda[\alpha] (\Omega)$, and $\robinlambda[\alpha] (\Omega) > 0$ as long as $\partial\Omega \neq \emptyset$.
at each vertex in $\partial\Omega$; see Section~\ref{sec:robin} for more details. We also define the corresponding Dirichlet ground state energy
\begin{equation}
\label{eq:rayleigh-quotient-dirichlet}
    \dirichletlambda (\Omega) = \inf\left\{\frac{\int_\Omega |f'(x)|^2\,\dx}{\int_\Omega |f(x)|^2\,\dx} : 0 \neq f \in H^1_0(\Omega) \right\},
\end{equation}
where $H^1_0 (\Omega) = \{ f \in H^1(\Omega): f(v)=0 \,\forall v \in \partial\Omega\}$.

In what follows, all partitions will be as in Definition~\ref{def:partition}.

\begin{definition}
Let $\Gamma$ be a connected, compact metric graph.
\begin{enumerate}
\item[{\bf (a)}] We define the \emph{Robin functional} with parameter $\alpha>0$, $\energyr[\alpha] : \partition \mapsto \R$, on the set of all partitions $\partition$ of $\Gamma$ by
\begin{equation}
\label{eq:energyr}
    \energyr[\alpha] (\partition) = \max_{i=1,\ldots,k} \robinlambda[\alpha] (\Omega_i),
\end{equation}
where $\partition = (\Omega_1,\ldots,\Omega_k)$.
\item[{\bf (b)}] The $k$-minimal Robin energy of $\Gamma$, $k \geq 1$, is given by
\begin{equation}
\label{eq:roptenergy}
    \roptenergy[\alpha]k (\Gamma) = \inf \{\energyr[\alpha] (\partitionk[k]) : \partitionk[k] \text{ is a $k$-partition of $\Gamma$} \}.
\end{equation}
\item[{\bf (c)}] The \emph{Dirichlet functional} $\energyd : \partition \mapsto \R$ is defined on partitions $\partition = (\Omega_1,\ldots,\Omega_k)$ of $\Gamma$ by
\begin{displaymath}
    \energyd (\partition) = \max_{i=1,\ldots,k} \dirichletlambda (\Omega_i).
\end{displaymath}
\item[{\bf (d)}] The $k$-minimal Dirichlet energy of $\Gamma$, $k \geq 2$, is given by
\begin{displaymath}
    \doptenergy[k] (\Gamma) = \inf \{ \energyd (\partitionk[k]) : \partitionk[k] \text{ is a $k$-partition of $\Gamma$} \}.
\end{displaymath}
\end{enumerate}
\end{definition}

The existence, for any compact $\Gamma$ and any $k \in \N$, of a minimal partition for the Dirichlet functional, was proved in \cite[Section~4]{KeKuLeMu21}, although this Dirichlet functional was first studied much earlier, in \cite{BBRS12}, where the authors showed that, under certain smoothness and genericity assumptions, its critical points correspond to eigenfunctions of the standard Laplacian on the whole graph.

We can now formulate our principal theorems, as noted in the introduction. First, we establish existence of a minimal Robin partition for each $\alpha>0$.  Second, we establish convergence of the normalised minimal partition energy values to their Cheeger counterparts as $\alpha \to 0$, together with convergence of the actual minimal partitions up to a subsequence. As $\alpha \to \infty$, the (non-normalised) minimal values and a subsequence of the minimising partitions converge to their Dirichlet counterparts.

\begin{theorem}[Existence and basic properties of minimal Robin partitions]
\label{thm:robin-existence}
Let $\Gamma$ be a connected, compact metric graph.
\begin{enumerate}
\item For each $k\geq 2$ and $\alpha>0$ there exists a $k$-partition $\partitionk[k]^\alpha$ of $\Gamma$ such that
\begin{equation}
\label{eq:roptenergy-attained}
    \roptenergy[\alpha]k (\Gamma) = \energyr[\alpha] (\partitionk[k]^\alpha).
\end{equation}
\item The function $\alpha \mapsto \roptenergy[\alpha]k (\Gamma)$ is a continuous and strictly monotonically increasing function of $\alpha>0$.
\end{enumerate}
\end{theorem}

\begin{theorem}[Limit cases and convergence to Cheeger partitions]
\label{thm:robin-limit}
Let $\Gamma$ be a connected, compact metric graph.
\begin{enumerate}
\item As $\alpha \to 0$,
\begin{equation}
\label{eq:roptenergy-convergence}
    \frac{\roptenergy[\alpha]k (\Gamma)}{\alpha} \to \cheeger[k] (\Gamma).
\end{equation}
Moreover, for any sequence (in $\alpha$) of minimal partitions $\partitionk[k]^\alpha$ with $\alpha \to 0$, there exists a $k$-partition $\partitionk[k]^C$ achieving equality in \eqref{eq:cheegerk} such that, up to a subsequence,
\begin{equation}
\label{eq:ropt-partition-convergence}
    \partitionk[k]^\alpha \to \partitionk[k]^C
\end{equation}
as $\alpha \to 0$, where partition convergence is in the sense described in Definition~\ref{def:partition-convergence} (which in particular implies convergence of each partition element with respect to Hausdorff distance).
\item As $\alpha \to \infty$,
\begin{equation}
\label{eq:roptenergy-dirichlet-convergence}
    \roptenergy[\alpha]k (\Gamma) \to \doptenergy[k] (\Gamma).
\end{equation}
Moreover, up to a subsequence there is convergence of the minimising partitions for $\roptenergy[\alpha]k$ to a minimising partition for $\doptenergy[k]$ in the same sense as in (1).
\end{enumerate}
\end{theorem}

\begin{remark}
If we define the size of the boundary $|\partial\Omega|$ simply as $\# \{v \in \Gamma : v \in \partial\Omega\}$, then the factor $\EffDeg (v)$ in \eqref{eq:robinlambda} should be replaced by $1$ (just as with the Cheeger constant); in this case, Theorem~\ref{thm:robin-limit}, suitably adjusted, will continue to hold, with the same proof.
\end{remark}

\begin{remark}
\label{rem:robin-p-type}
As with the Cheeger problem (cf.~Remark~\ref{rem:cheeger-p-type} and Appendix~\ref{app:cheeger-alt}), existence can be proved for (exhaustive or non-exhaustive) $k$-partitions minimising $p$-type functionals of the form
\begin{displaymath}
    \energyr[\alpha]_p (\partition) := \left[\sum_{i=1}^k \robinlambda[\alpha] (\Omega_i)^p \right]^{1/p}.
\end{displaymath}
Theorem~\ref{thm:robin-limit} should continue to hold, \textit{mutatis mutandis}, for these more general functionals, with an easily modified proof. We will not consider these functionals further, and do not go into details.
\end{remark}

\begin{remark}
If $\alpha<0$, then $\robinlambda[\alpha](\Omega) < 0$ is always bounded from above but not necessarily from below; the natural optimisation problem equivalent to \eqref{eq:roptenergy} becomes
\begin{displaymath}
    \sup_{\partition = (\Omega_1,\ldots,\Omega_k)} \min \{\robinlambda[\alpha] (\Omega_1), \ldots, \robinlambda[\alpha] (\Omega_k)\} \leq 0,
\end{displaymath}
or equivalently its negative,
\begin{displaymath}
    \inf_{\partition = (\Omega_1,\ldots,\Omega_k)} \max \{-\robinlambda[\alpha] (\Omega_1), \ldots, -\robinlambda[\alpha] (\Omega_k)\} \geq 0.
\end{displaymath}
Since the key expansion \eqref{eq:robin-at-0} (namely \eqref{eq:robin-expansion-at-zero-domains} for graphs) is two-sided, it should be possible to show that the latter quantity, when normalised by $\alpha$, converges to $\cheeger[k] (\Gamma)$ as $\alpha \nearrow 0$, with a corresponding statement holding for $p$-sums as in Remark~\ref{rem:robin-p-type}. The proofs should be a fairly easy generalisation of ours, but we will not explore this question further here.
\end{remark}

\section{Properties of the Robin Laplacian on metric subgraphs}
\label{sec:robin}

In this section we will collect some basic properties of the quantity \eqref{eq:robinlambda}, including well-posedness, properties of $\robinlambda[\alpha] (\Omega)$ in dependence on $\Omega$ for fixed $\alpha$, and properties in dependence on $\alpha$ for fixed $\Omega$. For the general framework, instead of considering $\EffDeg (v)\alpha$ we will consider a general boundary potential $\beta (v)$ which may depend on the vertex; thus in our concrete applications $\beta(v) = \EffDeg (v) \alpha$. To avoid confusion, wherever possible we will use $\beta$ for the general case, and $\EffDeg (v)\alpha$, or just $\alpha$, for the particular choice $\beta(v) = \EffDeg (v) \alpha$; however, in a slight abuse of notation, we will write $\robinlambda[\beta] (\Omega)$ for the ordered eigenvalues of the former, and $\robinlambda[\alpha] (\Omega)$ for the ordered eigenvalues of the latter.

\subsection{General framework and basic properties}

Given a connected subgraph $\Omega$ of $\Gamma$ with boundary $\partial\Omega$ consisting of $V$ distinct vertices in $\Omega$, for $\beta =\beta(v) \in \R^{V}$ we define the bilinear form $\robinform[\beta] : H^1(\Omega)\times H^1(\Omega) \to \R$ by
\begin{equation}
\label{eq:robinform}
    \robinform[\beta](f,g)=\int_{\Omega} f'(x)g'(x) \dx +\sum_{v \in \partial\Omega} \beta(v) f(v)g(v).
\end{equation}
It is well known (see, e.g., \cite[Sections~1.4.4, 2.5 and 3.1]{BeKu13}, \cite[Section~2]{BeKeKuMu19} or \cite{BaScSo22}) that the operator on $L^2(\Omega)$ associated with $\robinform[\beta]$ is the Laplacian satisfying a continuity condition at all vertices, a Robin condition
\begin{displaymath}
    \sum_{\substack{e \sim v\\ e \subset \Omega}} \partial_\nu f|_e (v)
    = \beta(v) f(v)
\end{displaymath}
at each boundary vertex $v \in \partial\Omega$ (where $\partial_\nu f|_e(v)$ is the derivative of $f|_e$ at $v$ pointing away from $v$, and only edges $e$ incident with $v$ in the subgraph $\Omega$ are counted) and Kirchhoff conditions, corresponding to $\beta=0$, at all internal vertices of $\Omega$.
We also speak of a \emph{$\delta$-potential} of strength $\beta(v)$ at $v$. The following basic properties are also well known (again, see \cite{BaScSo22,BeKeKuMu19,BeKu13} among many others; simplicity of the first eigenvalue and strict positivity of a corresponding eigenfunction are discussed in \cite{Ku19}).

\begin{lemma}
\label{lem:basic}
The operator on $L^2(\Omega)$ associated with \eqref{eq:robinform} is self-adjoint, coercive up to a multiple of the identity on $L^2$, and has compact resolvent. In particular, its eigenvalues form a sequence of real numbers
\begin{displaymath}
    \lambda_1 < \lambda_2 \leq \lambda_3 \leq \ldots \to \infty,
\end{displaymath}
repeated according to their finite multiplicities. The first eigenvalue $\lambda_1 = \robinlambda[\beta] (\Omega)$ is simple, admits an eigenfunction strictly positive in $\Omega = \overline{\Omega}$, and may be characterised by
\begin{equation}
\label{eq:rayleigh-quotient}
    \robinlambda[\beta] (\Omega) = \inf_{0\neq f\in H^1(\Omega)} \frac{\robinform[\beta](f,f)}{\|f\|_2^2} = 
    \inf_{0\neq f\in H^1(\Omega)} \frac{\int_\Omega |f'(x)|^2\,{\rm d}x
    +\sum_{v \in \partial\Omega} \beta(v)f(v)^2}{\int_\Omega |f(x)|^2\,{\rm d}x}.
\end{equation}
It is strictly positive if $\beta(v)>0$ for at least one $v \in \partial\Omega$.
\end{lemma}

\subsection{Continuity and monotonicity with respect to edge lengths for fixed $\alpha$}

We will be interested in how $\robinlambda[\alpha] (\Omega)$ depends on continuous deformations of $\Omega \subset \Gamma$, which amounts to dependence on the edge lengths. We formulate our result, which will include both continuity and monotonicity, for a family of (connected, compact) metric graphs having a fixed topology, or equivalently a fixed \emph{underlying discrete graph} (see Definition 2.1, and the discussion around it, in \cite{KeKuLeMu21}). We will, however, include the degenerate limit when an edge disappears. While the general conclusion in this case is well known, and in fact known in greater generality (see \cite{BeLaSu19}, or the related works \cite{Bo22,Ca19}), for the sake of completeness %and rigour 
we will include a precise formulation of the statement in the form we need it.%, together with an elementary proof which we defer to Appendix~\ref{sec:proof-continuity-monotonicity}.

More precisely, we consider the family $(\Gamma_t)_{t\geq 0}$ of similar graphs, which are identical to each other except for a distinguished edge $e_t$ of length $t \geq 0$. Denote by $v_\beta, v_\gamma$ the incident vertices with $e_t$, for $t>0$, and by $v_{\beta+\gamma}$ the vertex of the limit graph when $t=0$ (i.e. when $e_t$ is eliminated, the respective vertices incident with $e_t$ are ``glued'' together).

\begin{theorem}
\label{thm:continuity-monotonicity}
For the family $(\Gamma_t)_{t\geq 0}$ just described, denote by $\lambda_1^{\beta,\gamma}(\Gamma_t)$ the first eigenvalue with strength $\beta \geq 0$ at $v_\beta$ and strength $\gamma\geq 0$ at $v_\gamma$, and by $\lambda_1^{\beta+\gamma}(\Gamma_0)$ the eigenvalue at $\Gamma_0$ with strength $\beta+\gamma$ at $v_{\beta+\gamma}$; in all cases we assume fixed non-negative $\delta$-potentials at all other vertices. Then
\begin{displaymath}
    \lambda_1^{\beta,\gamma}(\Gamma_t) \nearrow \lambda_1^{\beta+\gamma}(\Gamma_0).
\end{displaymath}
\end{theorem}

Note that $\beta$ and $\gamma$ are allowed to be zero, and $v_\beta$ and $v_\gamma$ are allowed to have degree two; thus Theorem~\ref{thm:continuity-monotonicity} also covers the special case of shortening an edge without making it disappear in the limit; take $\gamma=0$ and $v_\gamma$ as a dummy vertex at a suitable position in the edge being shortened, as in Figure~\ref{fig:continuity}. In particular, we obtain a ``local'' continuity and monotonicity statement: assuming constant strengths $\alpha$, shortening any edge raises the first eigenvalue continuously and monotonically.
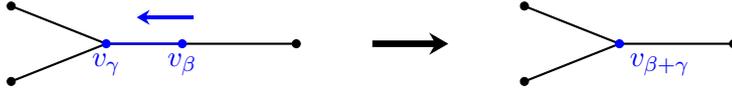
\begin{figure}[H]
\begin{tikzpicture}[scale=1]
\draw[thick] (2,0) -- (4.5,0);
\node at (3,0) [anchor=north] {{\color{blue}$v_\beta$}};
\node at (2,0) [anchor=north] {{\color{blue}$v_\gamma$}};
\draw[thick] (2,0) -- (0.75,0.5);
\draw[thick] (2,0) -- (0.75,-0.5);
\draw[fill] (0.75,0.5) circle (1.5pt);
\draw[fill] (0.75,-0.5) circle (1.5pt);
\draw[fill,blue] (2,0) circle (1.5pt);
\draw[fill,blue] (3,0) circle (1.5pt);
\draw[thick,blue] (2,0) -- (3,0);
\draw[fill] (4.5,0) circle (1.5pt);
\draw[-{Stealth[scale=0.5,angle'=60]},blue,line width=1.5pt] (3.15,0.35) -- (2.4,0.35);
\draw[-{Stealth[scale=0.5,angle'=60]},line width=2.5pt] (5.5,0) -- (6.5,0);
\draw[thick] (8.75,0) -- (10.25,0);
\node at (8.75,0) [anchor=north west] {{\color{blue}$v_{\beta+\gamma}$}};
\draw[thick] (8.75,0) -- (7.5,0.5);
\draw[thick] (8.75,0) -- (7.5,-0.5);
\draw[fill] (7.5,0.5) circle (1.5pt);
\draw[fill] (7.5,-0.5) circle (1.5pt);
\draw[fill,blue] (8.75,0) circle (1.5pt);
\draw[fill] (10.25,0) circle (1.5pt);
\end{tikzpicture}
\caption{Continuity of $\robinlambda[\beta,\gamma]$ with respect to edge lengths as a special case of Theorem~\ref{thm:continuity-monotonicity}. Here $\gamma$ may be taken equal to zero, in which case $v_\gamma$ is a dummy vertex and we recover a ``standard'' continuity result.}\label{fig:continuity}
\end{figure}

\begin{proof}
We first note that there is convergence of the eigenvalues in the situation described in the statement; this is an immediate consequence of the continuity property contained in \cite[Theorem~3.5 and Proposition~3.6]{BeLaSu19} (where their Lemma 3.4, and hence Condition 3.2, are clearly satisfied in our case). For monotonicity, it suffices to note that for any strictly positive $t>0$ the derivative $\frac{\d}{\d t} (\lambda_1^{\beta,\gamma}(\Gamma_t)$ (exists and) is positive. To see this, we fix $t>0$ and insert two dummy vertices $v_-$ and $v_+$ in the interior of $e_t$, to create a new edge entirely contained in the original $e_t$, which has Neumann--Kirchhoff vertex conditions at both endpoints. It now follows from the well-known Hadamard formula for the derivative of a simple eigenvalue with respect to the edge length of an edge with Neumann--Kirchhoff conditions (see \cite[Lemma~5.2]{BaLe17} or \cite[Remark~3.14]{BeKeKuMu19} and the references therein, and note that the proof in \cite{BaLe17} carries over with trivial adjustments to our case) that indeed $\frac{\d}{\d t} (\lambda_1^{\beta,\gamma}(\Gamma_t)$ is positive.
\end{proof}

We finish this subsection with a statement that we will need for the existence result, Theorem~\ref{thm:robin-existence}(1) (here we return to the case of fixed $\alpha$, although the statement holds more generally): if $\alpha>0$, $\robinlambda[\alpha] (\Omega_n) \to \infty$ whenever $|\Omega_n| \to 0$. We will formulate a slightly stronger version of this statement, which is however a fairly direct consequence of a Faber--Krahn (or Nicaise) type isoperimetric inequality established in \cite{KaKuTK16}.

\begin{lemma}
\label{lem:robin-volume-divergence}
Suppose $\Omega_n$ is a sequence of subgraphs of $\Gamma$ such that $|\Omega_n| \to 0$. Then, for fixed $\alpha > 0$, $\robinlambda[\alpha] (\Omega_n) \to \infty$. Moreover, for any sequence of parameters $\alpha_n > 0$ bounded from above (but $\alpha_n \to 0$ is allowed), the quotient $\frac{\lambda_1^{\alpha_n}(\Omega_n)}{\alpha_n} \to \infty$.
\end{lemma}

\begin{proof}
We first claim that, given $\alpha>0$, for any subgraph $\Omega \subset \Gamma$ for which $\partial\Omega \neq \emptyset$,
\begin{equation}
\label{eq:robin-useful-bound}
    \robinlambda[\alpha] (\Omega) \geq \frac{2\alpha\pi^2}{2|\Omega|(\pi^2 + 4\alpha|\Omega|)} \to \infty.
\end{equation}
To prove this, we first observe that since $\Omega \neq \Gamma$ it must have at least one Robin vertex. We may without loss of generality replace $\EffDeg_\Omega (v)$ by $1$ for all $v \in \partial\Omega$ since this can only reduce the eigenvalue. It thus follows from \cite[Theorem~1]{KaKuTK16} that
\begin{equation}
\label{eq:nicaise-robin}
    \robinlambda[\alpha] (\Omega) \geq \robinlambda[(\alpha,0)] (I),
\end{equation}
where $I$ is an interval of length $|\Omega|$, and $\robinlambda[(\alpha,0)](I)$ is its first eigenvalue with a Robin coupling of strength $\alpha$ at one endpoint and Neumann conditions at the other. By symmetry, this eigenvalue, in turn, is equal to the first eigenvalue of an interval of length $2|\Omega|$ with a Robin condition of strength $\alpha$ at both endpoints. Using known estimates for the latter eigenvalue (see \cite[Eq.~(A5)]{FrKe21}) leads to \eqref{eq:robin-useful-bound}.

Now the first statement of the lemma follows immediately from \eqref{eq:robin-useful-bound}, since if $|\Omega_n| \to 0$ (which in particular means $\partial\Omega_n \neq \emptyset$ for $n$ large enough) then the right-hand side of \eqref{eq:robin-useful-bound}, applied to $\Omega=\Omega_n$, diverges to $\infty$. Similarly, if both $\Omega_n$ and $\alpha_n$ are varying, but the $\alpha_n$ are bounded, then dividing both sides of \eqref{eq:robin-useful-bound} by $\alpha=\alpha_n$ yields
\begin{displaymath}
    \frac{\robinlambda[\alpha_n] (\Omega_n)}{\alpha_n} \geq \frac{2\pi^2}{2|\Omega_n|(\pi^2 + 4\alpha_n|\Omega_n|)} \to \infty
\end{displaymath}
as $|\Omega_n| \to 0$.
\end{proof}

\subsection{Properties of the Robin eigenvalue in dependence on $\alpha$}
\label{sec:robin-dependence-alpha}

We next collect some background results about the behaviour of $\robinlambda[\alpha](\Omega)$ as a function of the real parameter $\alpha \in \R$ on a fixed subgraph $\Omega$. These properties are very well known in the case of domains (e.g., \cite{AnFrKe13} or \cite{BuFrKe17}, among many others). The proofs are essentially unchanged on metric graphs, and versions of have appeared elsewhere (see \cite[Section~3.1]{BeKu13} or \cite[Section~2]{BaScSo22}); for the sake of completeness we include the proof of the result in the form in which we will need it.

\begin{lemma}
\label{lem:robin-asymptotic-zero}
Let $\Omega \subset \Gamma$ be a fixed, connected subgraph of $\Gamma$. The function $\alpha \mapsto \robinlambda[\alpha] (\Omega)$ (as defined by \eqref{eq:robinlambda}) is an analytic, strictly monotonically increasing, and concave function of $\alpha \in \R$. At $\alpha=0$ it admits the expansion
\begin{equation}
\label{eq:robin-at-0}
    \robinlambda[\alpha] (\Omega) = \alpha \frac{|\partial\Omega|}{|\Omega|} + \mathcal{O} (\alpha^2),
\end{equation}
while, as $\alpha \to \infty$,
\begin{equation}
\label{eq:robin-at-infinity}
    \robinlambda[\alpha] (\Omega) \to \dirichletlambda (\Omega)
\end{equation}
from below.
\end{lemma}

\begin{proof}
The family of forms $\robinform[\alpha]$ defined by \eqref{eq:robinform} is a holomorphic self-adjoint family of type (a) in the sense of \cite[Section VII.4.2]{Ka76}; since $\robinlambda[\alpha]$ is simple (Lemma~\ref{lem:basic}), it thus depends analytically on $\alpha \in \R$ by the results of \cite[Section VII.4.6]{Ka76}. That $\robinlambda[\alpha]$ is weakly monotonically increasing is an immediate consequence of the Rayleigh quotient \eqref{eq:rayleigh-quotient} being a monotonically increasing function of $\alpha$ for each fixed function $0 \neq u \in H^1(\Omega)$; moreover, it is concave as the infimum of a family of affine functions (in the parameter $\alpha$). For \eqref{eq:robin-at-0}, we note that in general
\begin{equation}
\label{eq:robin-derivative-general}
    \lambda'(\alpha) = \frac{\d}{\d \alpha} R_\alpha [\psi] = \frac{\d}{\d \alpha} \frac{\robinform[\alpha] (\psi,\psi)}{\|\psi\|_2^2},
\end{equation}
for any simple eigenvalue $\lambda$ with eigenfunction $\psi$, a formula which in our case reduces to
\begin{displaymath}
    \lambda'(\alpha) = \frac{\sum_{v \in \partial \Omega} \EffDeg_\Omega (v) \psi(v)^2}{\|\psi\|_2^2},
\end{displaymath}
see, e.g., \cite[Proposition~3.1.6]{BeKu13}. This formula, used at $\alpha=0$, where the eigenfunctions associated with $\lambda_1^0 (\Omega) = 0$ are constant, yields
\begin{equation}
\label{eq:robin-derivative-at-zero}
    \frac{\d}{\d\alpha} \robinlambda[\alpha] (\Omega)\Big|_{\alpha=0} = \frac{|\partial\Omega|}{|\Omega|}.
\end{equation}
The formula \eqref{eq:robin-at-0} now follows from the analyticity of $\robinlambda[\alpha]$. Finally, \eqref{eq:robin-at-0} implies that $\robinlambda[\alpha] (\Omega)$ is strictly monotonically increasing in a neighbourhood of $\alpha=0$. Its analyticity implies it is strictly monotonically increasing on $\R$.

The inequality
\begin{equation}
\label{eq:robin-dirichlet-comparison}
    \robinlambda[\alpha] (\Omega) \leq \dirichletlambda (\Omega)
\end{equation}
for any $\alpha \in \R$ follows immediately from the respective variational characterisations \eqref{eq:rayleigh-quotient} and \eqref{eq:rayleigh-quotient-dirichlet} and the inclusion $H^1_0 (\Omega) \subset H^1(\Omega)$. The convergence as $\alpha \to \infty$ is well known on domains, and can be proved by elementary variational means (cf.\ \cite[Section~4.4.1]{BuFrKe17}).

On graphs, it can also be obtained using the powerful analyticity results outlined in \cite[Section~2.5]{BeKu13}. Namely, following the formalism for describing vertex conditions outlined in \cite[Section~1.4]{BeKu13} (see Section 1.4.4 and (1.4.29) there in particular), at a given vertex $v$ we write the Robin condition (for given $\alpha > 0$) in the form $A_v F(v) + B_v F'(v)$, where $F$ is the vector reproducing the values of a given function at the endpoint of each edge incident with $v$, and $A_v$ and $B_v$ are the matrices
\begin{displaymath}
	A_v = \left(\begin{matrix} 1 & -1 & 0 & \ldots & 0 & 0\\ 0 & 1 & -1 & \ldots & 0 & 0\\ \vdots & \vdots & \vdots & \ddots & \vdots & \vdots \\
	0 & 0 & 0 & \ldots & 1 & -1\\ \sin \gamma & 0 & 0 & \ldots & 0 & 0\end{matrix}\right), \qquad
	B_v = \left(\begin{matrix} 0 & 0 & \ldots & 0\\ 0 & 0 & \ldots & 0\\ \vdots & \vdots & \ddots & \vdots\\
	\cos \gamma & \cos \gamma & \ldots & \cos \gamma \end{matrix}\right),
\end{displaymath}
where $\gamma$ is the unique number in $(0,\frac{\pi}{2})$ such that $\alpha = \tan \gamma$. Note that when $\gamma = \frac{\pi}{2}$ we recover Dirichlet conditions, while the matrices are analytic functions of $\gamma \in (0,\frac{\pi}{2}]$ (indeed, of $\gamma \in \R$). Thus the claimed convergence result follows from \cite[Theorem~2.5.4]{BeKu13} (which in fact shows analytic convergence of the eigenvalues and the corresponding spectral projections as functions of $\gamma$).
\end{proof}

\begin{remark}
\label{rem:derivative-bound}
It is an immediate consequence of \eqref{eq:robin-derivative-at-zero} and the concavity of the function $\alpha \mapsto \robinlambda[\alpha] (\Omega)$ that
\begin{displaymath}
    (0 \leq\,)\, \frac{\d}{\d\alpha} \robinlambda[\alpha] (\Omega) \leq \frac{|\partial\Omega|}{|\Omega|}
\end{displaymath}
for any $\alpha>0$ and any connected subgraph $\Omega \subset \Gamma$.
\end{remark}

\section{Existence, continuity and monotonicity: Proof of Theorem~\ref{thm:robin-existence}}
\label{sec:robin-existence}

To prove existence (Theorem~\ref{thm:robin-existence}(1)), as with the proof of Theorem~\ref{thm:cheeger}, it is sufficient to check that the two conditions of \cite[Theorem~3.13]{KeKuLeMu21} are satisfied, which we recall are:
\begin{enumerate}
\item $\energyr[\alpha] (\partition) \leq \liminf_{n\to\infty} \energyr[\alpha] (\partition_n)$ whenever $\partition_n \to \partition$ in the sense of the partition convergence described in Section~\ref{sec:convergence-description};
\item $\energyr[\alpha] (\partition_n) \to \infty$ if there are vanishing $\Omega_n \in \partition_n$, i.e., such that $|\Omega_n| \to 0$.
\end{enumerate}
But the lower semicontinuity property (1) is a direct consequence of Theorem~\ref{thm:continuity-monotonicity}, while the coercivity condition (2) is a direct consequence of the first assertion of Lemma~\ref{lem:robin-volume-divergence}. This immediately yields existence.

We now consider Theorem~\ref{thm:robin-existence}(2). The (weak) monotonicity claim there is an immediate consequence of the monotonicity of the map
\begin{displaymath}
    \alpha \mapsto R_\alpha [f] = \frac{\int_\Omega |f'(x)|^2\dx + \alpha \sum_{v \in \partial\Omega} \EffDeg_\Omega (v) f(v)^2}{\int_\Omega |f(x)|^2\dx}
\end{displaymath}
for each fixed $0 \neq f \in H^1(\Omega)$, together with the definition of $\roptenergy[\alpha]k$ as an inf-max-inf. For the rest, given any $\alpha>0$, denote by
\begin{displaymath}
    \partitionk[k]^\alpha := \{ \Omega_1^\alpha,\ldots,\Omega_k^\alpha \}
\end{displaymath}
a minimising $k$-partition for this $\alpha$.

\begin{lemma}
\label{lem:optimal-partition-derivative-alpha}
With the above notation, given any fixed $\alpha_0>0$ with corresponding minimising partition $\{ \Omega_1^{\alpha_0},\ldots,\Omega_k^{\alpha_0} \}$, for each $i=1,\ldots,k$ we have
\begin{equation}
\label{eq:optimal-partition-derivative-alpha}
    0 < \frac{\d}{\d\alpha} \robinlambda[\alpha] (\Omega_i^{\alpha_0})\Big|_{\alpha=\alpha_0}
    \leq \frac{2k\sum_{v \in \VertexSet(\Gamma)} \deg_\Gamma v}{\ell_{\max}}
\end{equation}
where $\ell_{\max} > 0$ is the length of any longest edge in $\Gamma$.
\end{lemma}

While the right-hand side in \eqref{eq:optimal-partition-derivative-alpha} will not be optimal in general, the point is that it is a constant $C(\Gamma,k)$ depending only on the graph $\Gamma$ and $k \geq 2$.

\begin{proof}
Note that $\partial\Omega_i^{\alpha_0} \neq 0$ since $k \geq 2$; hence $\frac{\d}{\d\alpha} \robinlambda[\alpha] (\Omega_i^{\alpha_0}) > 0$ for all $\alpha>0$ by \eqref{eq:robin-derivative-general}, since any first eigenfunction is strictly positive on $\Omega = \overline{\Omega}$ (as noted in Lemma~\ref{lem:basic}). For the upper bound, first note that, since $\partitionk[k]^{\alpha_0}$ is optimal, certainly
\begin{displaymath}
    \robinlambda[\alpha_0] (\Omega_i^{\alpha_0}) \leq \robinlambda[\alpha_0] (I)
\end{displaymath}
for all $i$, where $I$ is an interval of length $\frac{\ell_{\max}}{k}$ (since we can construct a test partition by partitioning the longest edge of $\Gamma$ into $k$ equal pieces). The Nicaise-type inequality \eqref{eq:nicaise-robin} from \cite[Theorem~1]{KaKuTK16}, plus the fact that $\robinlambda[\alpha_0] (I)$ is equal to the first eigenvalue on an interval of length $\frac{\ell_{\max}}{2k}$ with mixed Robin--Neumann conditions, implies that $|\Omega_i^{\alpha_0}| \geq \frac{\ell_{\max}}{2k}$. It now follows from Remark~\ref{rem:derivative-bound} that $\frac{\d}{\d\alpha}\robinlambda[\alpha] (\Omega_i^{\alpha_0})\big|_{\alpha=\alpha_0} \leq \frac{2k|\partial\Omega_i^{\alpha_0}|}{\ell_{\max}}$. We finish by using the trivial estimate $|\partial\Omega_i^{\alpha_0}| \leq \sum_{v \in \VertexSet(\Gamma)} \deg_\Gamma v$.
\end{proof}

It follows from Lemma~\ref{lem:optimal-partition-derivative-alpha} that, for any given $\alpha,\beta > 0$,
\begin{equation}
\label{eq:partition-lipschitz-alpha}
    |\energyr[\beta] (\partitionk[k]^{\alpha}) - \energyr[\alpha] (\partitionk[k]^{\alpha})| \leq C(\Gamma,k)|\beta-\alpha|.
\end{equation}
Continuity of $\alpha \mapsto \roptenergy[\alpha]k (\Gamma)$ now follows easily from \eqref{eq:partition-lipschitz-alpha} and weak monotonicity: first, as the infimum of continuous functions $\roptenergy[\alpha]k (\Gamma)$ is right-continuous. Second, if, for some $\alpha_0$, we have a violation of left-continuity, $\limsup_{\alpha \nearrow \alpha_0} \roptenergy[\alpha]k (\Gamma)=: \Lambda < \roptenergy[\alpha_0]k (\Gamma)$, then for $|\alpha-\alpha_0| < C(\Gamma,k)^{-1}(\roptenergy[\alpha_0]k (\Gamma) - \Lambda)$, it follows from \eqref{eq:optimal-partition-derivative-alpha} that
\begin{displaymath}
    \energyr[\alpha_0] (\partitionk[k]^{\alpha}) \leq \energyr[\alpha] (\partitionk[k]^{\alpha}) + C(\Gamma,k)(\alpha_0-\alpha) < \roptenergy[\alpha_0]k (\Gamma),
\end{displaymath}
a contradiction to the definition of $\roptenergy[\alpha_0]k (\Gamma)$. This establishes continuity.

For strict monotonicity, suppose for a contradiction that
\begin{displaymath}
    \energyr[\alpha] (\partitionk[k]^{\alpha}) = \roptenergy[\alpha]k (\Gamma) = \roptenergy[\alpha_0]k (\Gamma) = \energyr[\alpha_0] (\partitionk[k]^{\alpha_0})
\end{displaymath}
for some $\alpha < \alpha_0$. By \eqref{eq:optimal-partition-derivative-alpha}, we have
$\frac{\d}{\d\alpha}|_{\alpha=\alpha_0} \energyr[\alpha_0] (\partitionk[k]^{\alpha_0}) > 0$, whence $\energyr[\alpha] (\partitionk[k]^{\alpha_0}) < \energyr[\alpha_0] (\partitionk[k]^{\alpha_0}) = \energyr[\alpha] (\partitionk[k]^{\alpha})$, a contradiction to the minimality of $\partitionk[k]^{\alpha}$.

\section{Convergence of minimal Robin partitions: Proof of Theorem~\ref{thm:robin-limit}}
\label{sec:robin-cheeger}

The goal of this section is to prove Theorem~\ref{thm:robin-limit}, that is, to prove the convergence of the $\alpha$-normalised minimal Robin energies (and up to a subsequence, partitions) to their Cheeger counterparts as $\alpha \to 0$, as well as the convergence to optimal Dirichlet partitions as $\alpha \to \infty$.

\subsection{Proof of Theorem~\ref{thm:robin-limit}(1)}

We will divide the proof into three steps:
\begin{enumerate}
\item[\textit{Step 1.}] $\limsup_{\alpha \to 0} \frac{\roptenergy[\alpha]k (\Gamma)}{\alpha} \leq \cheeger[k] (\Gamma)$. This inequality will follow directly from a ``test partition'' argument and \eqref{eq:robin-at-0}.
\item[\textit{Step 2.}] Up to a subsequence in $\alpha$, as $\alpha \to 0$ any sequence of minimal $k$-partitions has a limit which is itself a $k$-partition.
\item[\textit{Step 3.}] $\liminf_{\alpha \to 0} \frac{\roptenergy[\alpha]k (\Gamma)}{\alpha} \geq \cheeger[k] (\Gamma)$, and the limit partition from Step 2 is a minimal Cheeger partition. This step requires a careful approximation argument constructing ``uniform exterior approximations'' to the optimal partition subgraphs and using local domain monotonicity.
\end{enumerate}
Throughout, we will denote by
\begin{displaymath}
    \partitionk[k]^C := \{\Omega_1^C,\ldots,\Omega_k^C\}
\end{displaymath}
a minimal Cheeger $k$-partition of $\Gamma$ and, as before, for each $\alpha>0$, by
\begin{displaymath}
    \partitionk[k]^\alpha := \{\Omega_1^\alpha, \ldots, \Omega_k^\alpha\}
\end{displaymath}
a minimal Robin $k$-partition of $\Gamma$ for the parameter $\alpha$. We will implicitly assume that $\alpha \to 0$ represents an arbitrary sequence of values converging to $0$.
\bigskip

\emph{Step 1.} For each fixed $\alpha>0$, by optimality of $\partitionk[k]^\alpha$, we have
\begin{equation}
\label{eq:energyr-partitionc}
    \energyr[\alpha] (\partitionk[k]^\alpha) \leq \energyr[\alpha] (\partitionk[k]^C).
\end{equation}
Now since $\partitionk[k]^C = \{\Omega_1^C,\ldots,\Omega_k^C\}$ is fixed, we may apply \eqref{eq:robin-at-0} to $\Omega_i^C$ for each $i$: as $\alpha \to 0$,
\begin{displaymath}
    \frac{\robinlambda[\alpha] (\Omega_i^C)}{\alpha} = \frac{|\partial\Omega_i^C|}{|\Omega_i^C|} + \mathcal{O}(\alpha).
\end{displaymath}
Since $k$ is fixed and finite,
\begin{displaymath}
    \frac{\energyr[\alpha] (\partitionk[k]^C)}{\alpha}
    = \frac{\max_i \robinlambda[\alpha] (\Omega_i^C)}{\alpha}
    = \max_i \frac{|\partial\Omega_i^C|}{|\Omega_i^C|} + \mathcal{O}(\alpha)
    \longrightarrow \energyc (\partitionk[k]^C) = \cheeger[k] (\Gamma)
\end{displaymath}
as $\alpha \to 0$. Dividing both sides of \eqref{eq:energyr-partitionc} by $\alpha>0$ and passing to the limit as $\alpha \to 0$ yields
\begin{displaymath}
    \limsup_{\alpha \to 0} \frac{\roptenergy[\alpha]k (\Gamma)}{\alpha} \leq \cheeger[k] (\Gamma).
\end{displaymath}
\bigskip

\emph{Step 2.} Topological compactness of the set of $k$-partitions of $\Gamma$, parametrised by the location of their boundary points (as described in Section~\ref{sec:convergence-description}, see also \cite[Section~3.2]{KeKuLeMu21}), implies that, up to a subsequence in $\alpha$ and possibly a relabelling of the $\Omega_i^\alpha$, for each $i=1,\ldots,k$, either there exists $\Omega_i^0$ of positive measure such that
\begin{displaymath}
    \Omega_i^\alpha \to \Omega_i^0
\end{displaymath}
in the sense of Definition~\ref{def:subgraph-convergence} (the ``convergence case'') or $|\Omega_i^\alpha| \to 0$ (the ``vanishing case'').

In the case of convergence to $\Omega_i^0$, the convergence of the edge lengths implies $|\Omega_i^\alpha| \to |\Omega_i^0|$, while we have $|\partial\Omega_i^0| \leq \liminf_{\alpha\to 0} |\partial\Omega_i^\alpha|$ by Lemma~\ref{lem:subgraph-convergence}. Moreover, the $\Omega_i^0$ necessarily have pairwise disjoint interiors because the same is true of the $\Omega_i^\alpha$.

We claim that the vanishing case cannot occur. Indeed, if for some $i$, $|\Omega_i^\alpha| \to 0$, then by Lemma~\ref{lem:robin-volume-divergence},
\begin{displaymath}
    \frac{\robinlambda[\alpha] (\Omega_i^\alpha)}{\alpha} \to \infty
\end{displaymath}
as $\alpha \to 0$. Since for any fixed partition $\partition = (\Omega_1,\ldots,\Omega_k)$ each quotient $\frac{\robinlambda[\alpha] (\Omega_i)}{\alpha}$ remains bounded as $\alpha \to 0$ by \eqref{eq:robin-at-0}, for $\alpha$ sufficiently small the partition $\partitionk[k]^\alpha$ cannot be minimal.

Summarising, for any sequence of $\alpha \to 0$, there exists a $k$-partition $\partitionk[k]^0 := \{ \Omega_1^0, \ldots, \Omega_1^k \}$ such that, up to a subsequence in $\alpha$, $\partitionk[k]^\alpha \to \partitionk[k]^0$ in the sense of Definition~\ref{def:partition-convergence}, and for each $i$,
\begin{displaymath}
    \frac{|\partial\Omega_i^0|}{|\Omega_i^0|} \leq \liminf_{\alpha \to 0} \frac{|\partial\Omega_i^\alpha|}{|\Omega_i^\alpha|}.
\end{displaymath}
\bigskip

\emph{Step 3.} Let $\alpha\to 0$ be any sequence. We claim that, for any subsequence, up to a further subsequence, if $\partitionk[k]^0$ is the limit $k$-partition from Step 2 for this subsubsequence, then
\begin{equation}
\label{eq:energyc-limit}
    \energyc (\partitionk[k]^0) \leq \liminf_{\alpha \to 0} \frac{\energyr[\alpha] (\partitionk[k]^\alpha)}{\alpha}
\end{equation}
as $\alpha \to 0$. We first show how \eqref{eq:energyc-limit} completes the proof of Theorem~\ref{thm:robin-limit}(1). First, we have that, by definition
\begin{displaymath}
    \cheeger[k] (\Gamma) \leq \energyc (\partitionk[k]^0),
\end{displaymath}
and thus, in combination with Step 1, we see that
\begin{equation}
\label{eq:energyc-joined-limit}
     \limsup_{\alpha \to 0} \frac{\energyr[\alpha] (\partitionk[k]^\alpha)}{\alpha} \leq \cheeger[k] (\Gamma) \leq \energyc (\partitionk[k]^0) \leq  \liminf_{\alpha \to 0} \frac{\energyr[\alpha] (\partitionk[k]^\alpha)}{\alpha},
\end{equation}
where, technically, the latter inequality is to be interpreted in the sense that given any subsequence of our original sequence, it is valid for some subsequence of that subsequence. However, the arbitrariness of the first subsequence implies via a standard and elementary analysis argument that it must hold for the entire sequence. Thus, in particular, there is equality everywhere in \eqref{eq:energyc-joined-limit} as $\alpha \to 0$. Moreover, \eqref{eq:energyc-joined-limit} implies that $\partitionk[k]^0$ is a minimal Cheeger $k$-partition, and we have already shown in Step 2 that up to a subsequence $\partitionk[k]^\alpha \to \partitionk[k]^0$.

Thus it remains to prove \eqref{eq:energyc-limit}. To this end, we assume that we have restricted to a subsequence in $\alpha$ such that, for each $i$, all $\Omega_i^\alpha$ are similar to each other in the sense of Definition~\ref{def:subgraph-similar}. In particular, $|\partial\Omega_i^\alpha| \in \N$ is constant in $\alpha$. We also suppose that, for these $\alpha$, the corresponding partitions $\partitionk[k]^\alpha$ converge to the $k$-partition $\partitionk[k]^0$ as in Step 2.

We fix $\varepsilon>0$, and choose an $\alpha_\varepsilon>0$, to be specified precisely later. We then define, for each $i$, the exterior approximation
\begin{displaymath}
    \Omega_i^\varepsilon := \overline{\bigcup_{0<\alpha<\alpha_\varepsilon} \Omega_i^\alpha},
\end{displaymath}
to be interpreted as a subgraph of $\Gamma$ where the union of the $\Omega_i^\alpha$ as \emph{subgraphs} is taken to preserve the topology of the $\Omega_i^\alpha$ as described in Section~\ref{sec:convergence-description}; for example, if $\Omega_i^\varepsilon$ includes a vertex of $\Gamma$ along at least two different edges, but which is not in any of the $\Omega_i^\alpha$, then these two edges are to be considered not joined at this vertex in $\Omega_i^\varepsilon$, even if they are in $\Gamma$ (also note that the $\Omega_i^\varepsilon$ will not, in general, be disjoint).

This means, in particular, that $|\partial\Omega_i^\varepsilon| = |\partial\Omega_i^\alpha|$ for all $\alpha < \alpha_\varepsilon$ and all $i=1,\ldots,k$. We also choose $\alpha_\varepsilon>0$ such that
\begin{displaymath}
    |\Omega_i^\varepsilon| < |\Omega_i^\alpha| + \varepsilon
\end{displaymath}
for all $0<\alpha<\alpha_\varepsilon$ and all $i=1,\ldots,k$, which is possible due to the convergence of the $\Omega_i^\alpha$ to $\Omega_i^0$.

Now by choice of $\Omega_i^\varepsilon$, Theorem~\ref{thm:continuity-monotonicity} implies $\robinlambda[\alpha] (\Omega_i^\alpha) \geq \robinlambda[\alpha] (\Omega_i^\varepsilon)$. It follows that, if $f_i^\varepsilon \in \mathcal{O}(\alpha)$ is the remainder function giving equality in \eqref{eq:robin-at-0} for $\Omega_i^\varepsilon$, then
\begin{displaymath}
    \frac{\robinlambda[\alpha] (\Omega_i^\alpha)}{\alpha} \geq \frac{\robinlambda[\alpha] (\Omega_i^\varepsilon)}{\alpha} \geq \frac{|\partial\Omega_i^\alpha|}{|\Omega_i^\alpha|-\varepsilon} + f_i^\varepsilon (\alpha)
\end{displaymath}
for all $0 < \alpha < \alpha_\varepsilon$. Since $f_i^\varepsilon (\alpha) \to 0$ as $\alpha \to 0$ (and $f_i^\varepsilon$ is a fixed function depending only on $i$ and our choice of $\alpha_\varepsilon$), we may find some $\tilde\alpha_\varepsilon \leq \alpha_\varepsilon$ such that $|f_i^\varepsilon (\alpha)| < \varepsilon$ for all $0 < \alpha < \tilde\alpha_\varepsilon$ and all $i=1,\ldots,k$.

Summarising, for our fixed subsequence of $\alpha$, given any $\varepsilon>0$ we have found some $\tilde\alpha_\varepsilon>0$ such that for all $0<\alpha<\tilde\alpha_\varepsilon$ and all $i=1,\ldots,k$,
\begin{displaymath}
    \frac{\robinlambda[\alpha] (\Omega_i^\alpha)}{\alpha} \geq \frac{|\partial\Omega_i^\alpha|}{|\Omega_i^\alpha|-\varepsilon} - \varepsilon.
\end{displaymath}
Letting $\varepsilon \to 0$ (and $\alpha\to 0$) yields \eqref{eq:energyc-limit} and hence completes the proof of Theorem~\ref{thm:robin-limit}(1).

\subsection{Proof of Theorem~\ref{thm:robin-limit}(2)}

The proof follows the same three steps as the proof of Theorem~\ref{thm:robin-limit}(1), so we only sketch them briefly. Throughout, we will denote by
\begin{displaymath}
    \partitionk[k]^D := \{\Omega_1^D,\ldots,\Omega_k^D\}
\end{displaymath}
a minimal Dirichlet $k$-partition of $\Gamma$ and, for each $\alpha > 0$, by
\begin{displaymath}
    \partitionk[k]^\alpha := \{\Omega_1^\alpha,\ldots,\Omega_k^\alpha\},
\end{displaymath}
as we did in the proof of Theorem~\ref{thm:robin-limit}(1). We will assume that $\alpha \to \infty$ represents an arbitrary sequence of values diverging to $\infty$.

\emph{Step 1:} A simple ``test partition'' argument using the upper bound $\robinlambda[\alpha] (\Omega_i^D) \leq \dirichletlambda (\Omega_i^D)$ for all $i$ shows that
\begin{displaymath}
    \roptenergy[\alpha]k (\Gamma) \leq \doptenergy[k] (\Gamma)
\end{displaymath}
for each $\alpha > 0$, and so $\limsup_{\alpha \to \infty}\roptenergy[\alpha]k (\Gamma) \leq \doptenergy[k] (\Gamma)$ as well.

\emph{Step 2:} Up to a subsequence, $\partitionk[k]^\alpha$ converges to a $k$-partition $\partitionk[k]^\infty = (\Omega_1^\infty,\ldots,\Omega_k^\infty)$ in the sense of Definition~\ref{def:partition-convergence}. Again we have convergence to an $m$-partition, $m \leq k$, and merely have to rule out the vanishing case; but $|\Omega_i^\alpha| \to 0$ for some $i$ is immediately ruled out by Lemma~\ref{lem:robin-volume-divergence} together with the result of Step 1 and the minimality of $\partitionk[k]^\alpha$.

\emph{Step 3:} We sketch that $\partitionk[k]^\infty$ is a minimal $k$-partition for $\energyd$. As earlier, it is certainly sufficient to show that, up to a further subsequence in $\alpha$,
\begin{equation}
\label{eq:energyd-from-above}
    \energyd (\partitionk[k]^\infty) \leq \liminf_{\alpha \to \infty} \energyr[\alpha] (\partitionk[k]^\alpha).
\end{equation}
We take a subsequence such that, for each $i$, $\Omega_i^\alpha$ has the same topology for all $\alpha$, fix $\varepsilon>0$, choose $\alpha_\varepsilon>0$, and define the exterior approximation
\begin{displaymath}
    \Omega_i^\varepsilon := \overline{\bigcup_{\alpha>\alpha_\varepsilon} \Omega_i^\alpha}
\end{displaymath}
analogously to the proof of Theorem~\ref{thm:robin-limit}(2). We also define $g_i^\varepsilon (\alpha) := \robinlambda[\alpha](\Omega_i^\varepsilon) - \dirichletlambda(\Omega_i^\varepsilon)$, so that $g_i^\varepsilon$ is a continuous function of $\alpha$ which converges to $0$ as $\alpha \to \infty$ by Lemma~\ref{lem:robin-asymptotic-zero}. The local monotonicity of Theorem~\ref{thm:continuity-monotonicity} applied to $\Omega_i^\alpha$ and $\Omega_i^\varepsilon$ then yields
\begin{displaymath}
    \robinlambda[\alpha](\Omega_i^\alpha) \geq \robinlambda[\alpha](\Omega_i^\varepsilon) = \dirichletlambda(\Omega_i^\varepsilon) + g_i^\varepsilon (\alpha)
\end{displaymath}
for all $\alpha > \alpha_\varepsilon$. Note that the Dirichlet eigenvalues also satisfy the same monotonicity and continuity properties as their Robin counterparts (and this is easier to prove, cf.\ \cite[Lemma~3.6]{KeKuLeMu21}); hence we may choose $\alpha_\varepsilon$ large enough that
\begin{displaymath}
    |\dirichletlambda (\Omega_i^\varepsilon) - \dirichletlambda (\Omega_i^\infty)| < \tfrac{\varepsilon}{2};
\end{displaymath}
we also take it large enough that $|g_i^\varepsilon (\alpha)| < \tfrac{\varepsilon}{2}$ for all $\alpha > \alpha_\varepsilon$. It follows that
\begin{displaymath}
    \robinlambda[\alpha] (\Omega_i^\alpha) \geq \dirichletlambda (\Omega_i^\infty) - \varepsilon
\end{displaymath}
for all $\alpha > \alpha_\varepsilon$. This establishes \eqref{eq:energyd-from-above} and hence completes the proof.

\appendix
\section{On the perimeter of a metric subgraph}
\label{sec:appendix}

In \cite[eq.~(2.5)]{Ma22}, the \emph{perimeter} of a subset $\Omega \subset \Gamma$ of a compact metric graph is defined as
\begin{equation}
\label{eq:perimeter}
    \perim (\Omega) = \sup \left\{ \left|\int_\Omega f'(x)\dx\right| :
    f \in \widetilde H^1(\Gamma),\, \sum_{v \in \VertexSet (\Gamma)}f|_e(v) = 0 \, \forall v \in \VertexSet (\Gamma), \, \|f\|_{L^\infty (\Gamma)} = 1 \right\},
\end{equation}
where $\widetilde H^1(\Gamma)=\bigoplus_{e \in \EdgeSet (\Gamma)} H^1(e)$, that is, the functions $f$ should \emph{not} be assumed continuous at the vertices. In this appendix we will not distinguish between subsets $\Omega$ and subgraphs $\Omega$, by always assuming that the subgraph $\Omega$ has maximal connectivity, that is, each vertex $v \in \VertexSet (\Gamma) \cap \partial\Omega$ has exactly one descendant. In this case we also do not need to distinguish between $v$ and its unique descendant.

If $\Omega$ is connected (or consists of finitely many connected components) -- probably the only case of interest for the Cheeger problem -- then the perimeter can be characterised quite simply. As always, we assume without loss of generality that, for any such set $\Omega$, its (necessarily finite) boundary $\partial\Omega$ consists only of vertices; in particular, this means that every edge of $\Gamma$ is contained either in $\Omega$ or its complement $\Gamma \setminus \Omega$ (apart possibly from the endpoints of the edge).

We recall the following notion from Definition~\ref{def:boundary}, which under our current assumptions becomes slightly simpler, and refer to the examples and discussion in Section~\ref{sec:example}.

\begin{definition}
Given a set $\Omega \subset \Gamma$ and a vertex $v \in \VertexSet (\Gamma)$,
\begin{enumerate}
\item[{\bf (a)}] we define its \emph{degree in $\Omega$}, $\deg_\Omega (v) \in [0,\deg_\Gamma (v)]$, to be the number of edges incident with $v$ which lie in $\Omega$ in a neighbourhood of $v$;
\item[{\bf (b)}] if $v \in \partial\Omega$, then we define its \emph{effective degree in $\Omega$} to be
\begin{equation}
\label{eq:eff-deg}
    \EffDeg_\Omega (v) := \min \{ \deg_\Omega (v),\, \deg_\Gamma (v) - \deg_\Omega (v) \}.
\end{equation}
\end{enumerate}
\end{definition}

\begin{proposition}
\label{prop:perimeter-characterisation}
Let $\Omega \subset \Gamma$ be a connected subset of a compact graph $\Gamma$. Then its perimeter as defined by \eqref{eq:perimeter} satisfies
\begin{equation}
\label{eq:perimeter-characterisation}
    \perim (\Omega) = |\partial\Omega|
    = \sum_{v \in \partial\Omega} \EffDeg_\Omega (v)
    = \sum_{v \in \partial\Omega} \min \{ \deg_\Omega (v),\, \deg_\Gamma (v) - \deg_\Omega (v) \}.
\end{equation}
\end{proposition}

We stress that under our assumptions the sum in \eqref{eq:perimeter-characterisation} is finite; in fact, the reasoning in the proof shows that a set (at least, the closure of an open set) can have finite perimeter \emph{only if} it has a finite number of connected components.

\begin{proof}
Let $f \in \widetilde H^1(\Gamma)$ satisfy the Kirchhoff condition $\sum_{v \in \VertexSet (\Gamma)}f|_e(v) = 0$ $\forall v \in \VertexSet (\Gamma)$. Applying Green's formula \cite[eq.\ (2,4)]{Ma22} to $f$ and $\chi_\Omega \in \BV (\Gamma)$ yields
\begin{displaymath}
    \int_\Omega f'(x)\,\dx = \int_\Gamma \chi_\Omega f'
    = \sum_{v \in \VertexSet (\Gamma)} \sum_{e \sim v} f|_e (v)\chi_\Omega|_e (v)
    - \int_\Gamma fD\chi_\Omega.
\end{displaymath}
Now the latter integral is $0$ since $\chi_\Omega$ is constant on each edge (thanks to our assumption that its boundary consists only of vertices). Using the Kirchhoff condition satisfied by $f$ on the sum (which ensures that $\sum_{e \sim v} f|_e(v) \chi_\Omega|_e (v)$ is zero at any vertex $v$ \emph{not} in $\partial \Omega$), this identity reduces to
\begin{displaymath}
    \int_\Omega f'(x)\,\dx = \sum_{v \in \partial\Omega} \sum_{\substack{e \sim v\\ e \subset \Omega}} f|_e (v).
\end{displaymath}
Since all sums involved are finite, this means we may write the perimeter as
\begin{multline*}
    \perim (\Omega) = \sum_{v \in \partial\Omega} \sup \Bigg\{\Big| \sum_{\substack{e \sim v\\ e \subset \Omega}} f|_e(v)\Big|:\\ f \in \widetilde H^1(\Gamma),\, \sum_{v \in \VertexSet (\Gamma)}f|_e(v) = 0 \, \forall v \in \VertexSet (\Gamma), \, \|f\|_{L^\infty (\Gamma)} = 1 \Bigg\}.
\end{multline*}
Now for each $v \in \partial\Omega$, the Kirchhoff condition on $f$ implies that
\begin{displaymath}
    \sum_{\substack{e \sim v\\ e \subset \Omega}} f|_e(v) = -\sum_{\substack{e \sim v\\ e \subset \Gamma\setminus\Omega}} f|_e(v);
\end{displaymath}
under the condition that $\|f\|_{L^\infty (\Gamma)}=1$ it follows that
\begin{displaymath}
    \Big| \sum_{\substack{e \sim v\\ e \subset \Omega}} f|_e(v)\Big|
    \leq \min \Bigg\{\sum_{\substack{e \sim v\\ e \subset \Omega}} 1,\, \sum_{\substack{e \sim v\\ e \subset \Gamma \setminus\Omega}} 1\Bigg\}
    = \min \left\{\deg_\Omega (v), \deg_\Gamma (v) - \deg_\Omega (v) \right\}.
\end{displaymath}
This proves the inequality ``$\leq$'' in \eqref{eq:perimeter-characterisation}. For the other direction, since the number of vertices is finite it clearly suffices to prove that, for a given vertex $v \in \partial\Omega$ and $\varepsilon>0$ small enough that $B_\varepsilon(v) \subset \Gamma$ contains no other vertices of $\Gamma$, there exists some $f \in \widetilde H^1(\Gamma)$ with $\supp f \subset \overline{B_\varepsilon (v)}$, $\sum_{e \sim v} f|_e (v) = 0$, $\|f\|_{L^\infty(\Gamma)} = 1$, and
\begin{displaymath}
    \sum_{\substack{e\sim v\\ e \subset \Omega}} f|_e (v) \geq \EffDeg_\Omega (v).
\end{displaymath}
But this is trivial: by assumption there exist at least $\EffDeg_\Omega(v)$ edges in $B_\varepsilon (v) \cap \Omega$ and $\EffDeg_\Omega (v)$ edges in $B_\varepsilon (v) \cap (\Gamma\setminus\Omega)$. On each of the former we construct $f|_e$ to be linear, rising from $0$ at $\partial B_\varepsilon (v)$ to $1$ at $v$; on each of the latter we take $f|_e$ to fall from $0$ at $\partial B_\varepsilon (v)$ to $-1$ at $v$. On any other edges in $B_\varepsilon (v)$ (which is the case if and only if $\min \left\{\deg_\Omega (v), \deg_\Gamma (v) - \deg_\Omega (v) \right\} < \max \left\{\deg_\Omega (v), \deg_\Gamma (v) - \deg_\Omega (v) \right\}$), $f$ is taken to be identically zero on them. This $f$ has all the desired properties, establishing equality in \eqref{eq:perimeter-characterisation}.
\end{proof}

\section{Alternative definitions of the Cheeger constant and the size of the boundary}
\label{sec:cheeger-alt}

In this appendix we discuss a possible alternative definition of Cheeger partitions, which is used on domains in \cite{BoPa18}, and which is thus worth mentioning, and also a possible alternative definition of the size of the boundary.

\subsection{On defining Cheeger partitions}
\label{app:cheeger-alt}

Following \cite{BoPa18}, for our alternative definition we first define the $1$-Cheeger constant of a \emph{set} $\Omega \subset \Gamma$ as 
\begin{equation}
\label{eq:cheeger1}
	h_1 (\Omega) := \inf \left\{ \frac{|\partial E|}{|E|} : E \subset \Omega \right\},
\end{equation}
with $\partial E$ being the topological boundary of $E$ as a subset of $\Gamma$ (and where without loss of generality $E$ is assumed to have finite perimeter), see also \cite[Section~3]{Ma22} where this definition is made for graphs. If $\Omega = \Gamma$, then this coincides with our Definition~\ref{def:cheeger}(b) when $k=1$, since our partitions are not required to be exhaustive. In particular, in general the infimum in \eqref{eq:cheeger1} will give a value smaller than the isoperimetric ratio $\frac{|\partial\Omega|}{|\Omega|}$ of $\Omega$ (indeed, it may be strictly smaller; see \cite[Example 3.1]{Ma22}). One then considers the minimisation problem
\begin{equation}
\label{eq:cheeger-variant}
    \inf \{ \max \{h_1(E_1),\ldots,h_1(E_k)\}:
    (E_1,\ldots,E_k) \text{ is a $k$-partition of $\Gamma$} \} > 0,
\end{equation}
that is, instead of minimising the maximum of the isoperimetric ratios of the $\Omega_i$ one minimises the maximum of the respective isoperimetric ratios of the ``best possible subsets'' of them.

Within the class of \emph{non-exhaustive partitions}, as already noted in \cite{BoPa18} in the case of domains, these two minimisation problems, \eqref{eq:cheegerk} and \eqref{eq:cheeger-variant}, are equivalent. To see this, first note that $\cheeger[k](\Gamma)$ is clearly no smaller than the infimum in \eqref{eq:cheeger-variant}, since $\cheeger[1](\Omega) \leq \frac{|\partial\Omega|}{|\Omega|}$ for all $\Omega \subset \Gamma$. For the other direction, given a $k$-partition $\partition = (\Omega_1,\ldots,\Omega_k)$, we may without loss of generality replace each $\Omega_i$ with the subgraph $E_i \subset \Omega_i$ realising
\begin{displaymath}
    \frac{|\partial E_i|}{|E_i|} = \inf_{E \subset \Omega_i}
    \frac{|\partial E|}{|E|} = h_1 (\Omega_i);
\end{displaymath}
the existence of such a set is an easy adaptation of the proof of Theorem~\ref{thm:cheeger} or of the theory in \cite[Section~5]{Ma22}. This results in a new partition $\widetilde\partition = (E_1,\ldots,E_k)$ such that
\begin{displaymath}
    \max_{i=1,\ldots,k} \cheeger[1] (\Omega_i) = \max_{i=1,\ldots,k} \frac{|\partial E_i|}{|E_i|} \geq \cheeger[k](\Gamma).
\end{displaymath}
This proves equality between \eqref{eq:cheegerk} and \eqref{eq:cheeger-variant}.

Recall (cf.\ \cite[Section 3]{Ma22}) that a set $\Omega$ is called \emph{calibrable} if $h_1(\Omega) = \frac{|\partial\Omega|}{|\Omega|}$: the above argument confirms that in the case of non-exhaustive partitions, any minimising partition may be assumed to consist of calibrable subgraphs of $\Gamma$, and thus the two problems are equivalent.

If, however, we assume that our partitions are exhaustive, then one can show that Theorem~\ref{thm:cheeger} continues to hold, with the same proof (since \cite[Theorem~3.13]{KeKuLeMu21} covers the case of exhaustive partitions); however, \eqref{eq:cheegerk} and \eqref{eq:cheeger-variant} will not be equivalent in general. We will not go into details.

\subsection{On the size of the boundary}
\label{app:alternative-boundary}

Here we make a few remarks about a possible alternative definition of the size of the boundary of a subgraph $\Omega$ of $\Gamma$ which we could equally use: we could simply define
\begin{equation}
\label{eq:alternative-boundary}
    |\partial\Omega| := \# \VertexSet(\Gamma) \cap \partial\Omega
\end{equation}
to be the number of boundary points of $\Omega$ as a subset of $\Gamma$, where each point simply counts one towards the total regardless of its degree; this boundary size is clearly smaller than ours. By way of example, in Figure~\ref{fig:mutant-3-cut} we would now have $\frac{|\partial\Omega_3|}{|\Omega_3|} = \frac{1}{2}$, meaning there is a two-sided discontinuity: $\frac{|\partial\Omega_3|}{|\Omega_3|} < \lim_{n\to\infty} \frac{|\partial\Omega_{3,n}|}{|\Omega_{3,n}|}$ \emph{and} $\frac{|\partial\Omega_3|}{|\Omega_3|} < \lim_{n\to\infty} \frac{|\partial\widetilde \Omega_{3,n}|}{|\widetilde \Omega_{3,n}|}$, whereas with the old definition we are guaranteed to have equality in one of the two case.

However, with this alternative, simpler definition, we would still obtain Theorem~\ref{thm:cheeger}, since $\energyc$ would still be lower semicontinuous in general (as the above example illustrates, and as can be shown via an easy modification of the proof of Lemma~\ref{lem:subgraph-convergence}, which actually becomes simpler). This would lead to a different notion of which graphs are highly connected. For example, for the \emph{pumpkin chain} $\Gamma$ of Figure~\ref{fig:simple-pumpkin-chain} (assuming all edges to have length one), in either case the $2$-Cheeger cut of $\Gamma$ will split apart the two pumpkins at $v$. But with our definition of $|\partial\Omega|$, the effective degree of $v$ in each pumpkin will be $4$ and the Cheeger constant is
\begin{displaymath}
    h_2(\Gamma) = \max \left\{\tfrac{4}{5},\tfrac{4}{4}\right\} = 1;
\end{displaymath}
while with the alternative definition \eqref{eq:alternative-boundary} of $|\partial\Omega|$ we would have the much smaller value $h_2(\Gamma) = \max \{\frac{1}{5},\frac{1}{4}\} = \frac{1}{4}$. (Note, in either case, that the partition energy $\energyc$ will have a jump discontinuity if we shift the location of the cut into the $5$-pumpkin: it is lower semicontinuous but not continuous at $v$. While for this graph the two definitions result in the same Cheeger cut, by playing with the edge lengths it should be possible to generate an example where the optimal cut depends on our definition of $|\partial\Omega|$.)
\begin{figure}[ht]
\begin{tikzpicture}[scale=1.2]
\draw[thick,bend left=90]  (2,0) edge (4,0);
\draw[thick,bend right=90]  (2,0) edge (4,0);
\draw[thick,bend left=25] (2,0) edge (4,0);
\draw[thick, bend right=25] (2,0) edge (4,0);
\draw[fill] (2,0) circle (1.5pt);
\draw[fill] (4,0) circle (1.5pt);
\draw[fill] (0,0) circle (1.5pt);
\draw[thick] (0,0) -- (2,0);
\draw[thick,bend left=90] (0,0) edge (2,0);
\draw[thick,bend right=90] (0,0) edge (2,0);
\draw[thick,bend left=35] (0,0) edge (2,0);
\draw[thick,bend right=35] (0,0) edge (2,0);
\node at (2,0.25) [anchor=south] {$v$};
\end{tikzpicture}
\caption{A pumpkin chain $\Gamma$ consisting of a $5$-pumpkin and a $4$-pumpkin glued at the vertex $v$. All edges are assumed to have length~$1$.}
\label{fig:simple-pumpkin-chain}
\end{figure}
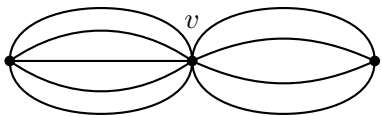

Both definitions (namely the smaller value \eqref{eq:alternative-boundary} and the larger value of Definition~\ref{def:boundary}(c)) have advantages and disadvantages: we would argue that there is no one canonical definition of $|\partial\Omega|$. For the graph of Figure~\ref{fig:simple-pumpkin-chain} we need only a rank-one cut (cf.\ \cite[Definition~2.2]{HoKe21}) in order to disconnect $\Gamma$ at $v$, which is consistent with assigning the vertex a ``size'' of $1$. However, the larger value coming from ``effective degree'' has several other advantages, which is why we use it: first, it corresponds to the notion of perimeter, as discussed in Appendix~\ref{sec:appendix}; second, in terms of the Laplacian, $\Gamma$ behaves like a highly connected graph (for example, it satisfies the bound of Band--L\'evy for doubly-connected graphs \cite[Theorem~2.1]{BaLe17}); and third, all our proofs may be trivially generalised if one replaces Definition~\ref{def:boundary}(c)) with the smaller value \eqref{eq:alternative-boundary} (while in the other direction we expect this would be less obvious).

\bibliographystyle{plain}

\begin{thebibliography}{99}

\bibitem{AnFrKe13}
P.~R.~S.~Antunes, P.~Freitas and J.~B.~Kennedy, \emph{Asymptotic behaviour and numerical approximation of optimal eigenvalues of the Robin Laplacian}, ESAIM: Control Optim.\ Calc.\ Var.\ \textbf{19} (2013), 438--459.

\bibitem{BBRS12}
R.~Band, G.~Berkolaiko, H.~Raz and U.~Smilansky, \emph{The number of nodal domains on quantum graphs as a stability index of graph partitions}, Comm.\ Math.\ Phys.\ \textbf{311} (2012), 815--838.

\bibitem{BaLe17}
R.~Band and G.~L\'evy, \emph{Quantum graphs which optimize the spectral gap}, Ann.\ Henri Poincar\'e \textbf{18} (2017), 3269--3323.

\bibitem{BaScSo22}
R.~Band, H.~Schanz and G.~Sofer, \emph{Differences between Robin and Neumann eigenvalues on metric graphs}, Ann.\ Henri Poincar\'e (2023), available online.

\bibitem{BeKeKuMu19}
G.~Berkolaiko, J.~B.~Kennedy, P.~Kurasov and D.~Mugnolo, \emph{Surgery principles for the spectral analysis of quantum graphs}, Trans.\ Amer.\ Math.\ Soc.\ \textbf{372} (2019), 5153--5197.

\bibitem{BeKu13}
G.~Berkolaiko and P.~Kuchment, \emph{Introduction to Quantum Graphs}. Math.\ Surveys and Monographs, vol.\ 186. Amer.\ Math.\ Soc., Providence, RI, 2013.

\bibitem{BeKu12}
G.~Berkolaiko and P.~Kuchment, \emph{Dependence of the spectrum of a quantum graph on vertex conditions and edge lengths}. Spectral geometry, 117--137, Proc.\ Sympos.\ Pure Math., vol.\ 84. Amer.\ Math.\ Soc., Providence, RI, 2012.

\bibitem{BeLaSu19}
G.~Berkolaiko, Yu.~Latushkin and S.~Sukhtaiev, \emph{Limits of quantum graph operators with shrinking edges}, Adv.\ Math.\ \textbf{352} (2019), 632--669.

\bibitem{BiKe22}
P.~Bifulco and J.~Kerner, \emph{Comparing the spectrum of Schrödinger operators on quantum graphs}, Proc.\ Amer.\ Math.\ Soc.\ \textbf{152} (2024), 295--306.

\bibitem{BoPa18}
V.~Bobkov and E.~Parini, \emph{On the higher Cheeger problem}, J.\ London Math.\ Soc.\ (2) \textbf{97} (2018), 575--600.

\bibitem{BoBuFr20}
B.~Bogosel, D.~Bucur and I.~Fragal\`a, \emph{Phase field approach to optimal packing problems and related Cheeger clusters}, Appl.\ Math.\ Optim.\ \textbf{81} (2020), 63--87.

\bibitem{BNH17}
V.~Bonnaillie-Noël and B.~Helffer, \emph{Nodal and spectral minimal partitions -- The state of the art in 2016}, Chapter 10 in A.~Henrot (ed.), \emph{Shape optimization and spectral theory}, De Gruyter Open, Warsaw, 2017.

\bibitem{Bo22}
D.~I.~Borisov, \emph{Analyticity of resolvents of elliptic operators on quantum graphs with small edges}, Adv.\ Math.\ \textbf{397} (2022), 108125.

\bibitem{BrTaChSz14}
X.~Bresson, X.-C.~Tai, T.~F.~Chan and A.~Szlam, \emph{Multi-class transductive learning based on $\ell^1$ relaxations of Cheeger cut and Mumford-Shah-Potts Model}, J.\ Math.\ Imaging Vision \textbf{49} (2014), 191--201.

\bibitem{BuFr19adv}
D.~Bucur and I.~Fragal\`a, \emph{Proof of the honeycomb asymptotics for optimal Cheeger clusters}, Adv.\ Math.\ \textbf{350} (2019), 97--129.

\bibitem{BuFr19ccm}
D.~Bucur and I.~Fragal\`a, \emph{On the honeycomb conjecture for Robin Laplacian eigenvalues}, Commun.\ Contemp.\ Math.\ \textbf{21} (2019), 1850007.

\bibitem{BuFrGi18}
D.~Bucur, I.~Fragal\`a and A.~Giacomini, \emph{Optimal partitions for Robin Laplacian eigenvalues}, Calc.\ Var.\ Partial Differential Equations \textbf{57} (2018), 122.

\bibitem{BuFrKe17}
D.~Bucur, P.~Freitas and J.~B.~Kennedy, \emph{The Robin problem}, Chapter 4 in A.~Henrot (ed.), \emph{Shape optimization and spectral theory}, De Gruyter Open, Warsaw, 2017.

\bibitem{Ca19}
C.~Cacciapuoti, \emph{Scale invariant effective Hamiltonians for a graph with a small compact core}, Symmetry \textbf{11} (2019), 359.

\bibitem{Ch70}
J. Cheeger, \emph{A lower bound for the smallest eigenvalue of the Laplacian},
Problems in analysis: A symposium in honor of Salomon Bochner (1970), 195--199.

\bibitem{Ch05}
F.~Chung, \emph{Laplacians and the Cheeger inequality for directed graphs}, Ann.\ Comb.\ \textbf{9} (2005), 1--19.

\bibitem{FrKe21}
P.~Freitas and J.~B.~Kennedy, \emph{Extremal domains and Pólya-type inequalities for the Robin Laplacian on rectangles and unions of rectangles}, Int.\ Math.\ Res.\ Not.\ IMRN (2021), 13730--13782.

\bibitem{HoKe21}
M.~Hofmann and J.~B.~Kennedy, \emph{Interlacing and Friedlander-type inequalities for spectral minimal partitions of metric graphs}, Lett.\ Math.\ Phys.\ \textbf{111} (2021), 96.

\bibitem{KaKuTK16}
G.~Karreskog, P.~Kurasov and I.~Trygg Kupersmidt, \emph{Schrödinger operators on graphs: symmetrization and Eulerian cycles}, Proc.\ Amer.\ Math.\ Soc.\ \textbf{144} (2016), 1197--1207.

\bibitem{Ka76}
T.~Kato, \emph{Perturbation Theory for Linear Operators}, Grundlehren der mathematischen Wissenschaften vol.\ 132, 2nd edition, Springer-Verlag, Berlin-Heidelberg, 1976.

\bibitem{KeKuLeMu21}
J.~B.~Kennedy, P.~Kurasov, C.~L\'ena and D.~Mugnolo, \emph{A theory of spectral partitions of metric graphs}, Calc.\ Var.\ Partial Differential Equations \textbf{60} (2021), 61.

\bibitem{Ku19}
P.~Kurasov, \emph{On the ground state for quantum graphs}, Lett.\ Math.\ Phys.\ \textbf{109} (2019), 2491--2512.

\bibitem{Ma23}
J.~M.~Maz\'on, \emph{The total variational flow in metric graphs}, Mathematics in Engineering \textbf{5} (2023), 1--38.

\bibitem{Ma22}
J.~M.~Maz\'on, \emph{The Cheeger cut and Cheeger problem in metric graphs}, Anal.\ Math.\ Phys.\ \textbf{12} (2022), 117.

\bibitem{Mu19}
D.~Mugnolo, \emph{What is actually a metric graph?}, preprint, 2019, arXiv:1912.07549.

\bibitem{Pa11}
E.~Parini, \emph{An introduction to the Cheeger problem}, Surveys Math.\ Appl.\ \textbf{6} (2011), 9--22.
\end{thebibliography}

\end{document}